\documentclass[a4paper, reqno, 12pt]{amsart}

\usepackage{geometry}       

\usepackage{float}
\usepackage{bm}
\usepackage{fullpage,xcolor}
\usepackage[mathscr]{euscript}
\usepackage[all]{xy}
\usepackage{epsfig}
\usepackage{amsfonts}
\usepackage{mathptmx}

\usepackage{graphicx}
\usepackage{amssymb} 
\usepackage{amsmath}
\usepackage{amsthm}
\usepackage{mathrsfs}
\usepackage{epstopdf}
\usepackage{url}
\usepackage[msc-links,alphabetic]{amsrefs}
\usepackage{tikz}

\usepackage{color}
\makeatletter

\@addtoreset{equation}{section}
\makeatother 

\theoremstyle{plain}
\newtheorem{theorem}{Theorem}[section]

\newtheorem{proposition}[theorem]{Proposition}

\theoremstyle{definition}
\newtheorem{definition}[theorem]{Definition}
\newtheorem{example}[theorem]{Example}
\newtheorem{remark}[theorem]{Remark}
\newtheorem{remarks}[theorem]{Remarks}

\usepackage{latexsym}
 
\title[From planar to annular to toroidal pseudo bracket polynomials]
 {From planar to annular to toroidal \\ bracket polynomials for pseudo knots and links}

\author{Ioannis Diamantis}
\address{Department of Data Analytics and Digitalisation,
Maastricht University, School of Business and Economics,
P.O.Box 616, 6200 MD, Maastricht,
The Netherlands.}
\email{i.diamantis@maastrichtuniversity.nl}

\author{Sofia Lambropoulou}
\address{School of Applied Mathematical and Physical Sciences, National Technical University of Athens, Zografou campus, GR-15780 Athens, Greece.}
\email{sofia@math.ntua.gr}
\urladdr{http://www.math.ntua.gr/~sofia}

\author{Sonia Mahmoudi}
\address{Advanced Institute for Materials Research, Tohoku University, 2-1-1 Katahira, Aoba-ku, Sendai 980-8577, Japan; RIKEN iTHEMS, 2-1 Hirosawa, Wako, Saitama 351-0198, Japan}
\email{sonia.mahmoudi@tohoku.ac.jp}



\date{}

\begin{document}

\setcounter{section}{-1}

\begin{abstract}
Pseudo links are equivalence classes under Reidemeister-type moves of link diagrams containing crossings with undefined over and under information. In this paper, we extend the Kauffman bracket and Jones-type polynomials from planar pseudo links to annular and toroidal pseudo links and their respective lifts from the three-space to the solid torus and the thickened torus. Moreover, since annular and toroidal pseudo links can be represented as mixed links in the three-sphere, we also introduce the respective Kauffman bracket and Jones-type polynomials for their planar mixed link diagrams. Our work provides new tools for the study of annular and toroidal pseudo links. 
\end{abstract}

\maketitle


\section{Introduction}\label{sec:0}

The study of pseudo knots was introduced by Hanaki in \cite{H} to model complex knotted structures like DNA knots, where advanced microscopy cannot detect the over/under information at certain crossings. This diagrammatic theory extends classical knot theory by providing a framework to study knots and links with undefined entanglements that arise in biological and physical systems.

A \textit{pseudo knot or link diagram} is defined as a planar knot or link diagram containing also double points missing over/under information, called \textit{precrossings}. These diagrams are classified under generalized Reidemeister moves (see Figure~\ref{reid}) that take into account the indeterminate nature of the precrossings. The equivalence class of a pseudo link diagram is called a \textit{pseudo link} (see \cites{DLM3, HJMR} for more details). As presented in our previous work \cite{DLM3}, planar, annular and toroidal pseudo link diagrams can lift to spatial pseudo links in the three-sphere or three-space, pseudo links in the solid torus, and pseudo links in the thickened torus, respectively. This allows us to translate equivalence of pseudo links diagrams into isotopy of their corresponding lifts. It also allows us to represent pseudo links in the solid torus and in the thickened torus by spatial pseudo mixed links. Spatial pseudo (mixed) links can be viewed as `knotted' embedded graphs in three-space, where the four-valent vertices may be interpreted as precrossings, and whose equivalence  and invariants were  studied in \cite{K3}. 

In this paper, we introduce the \textit{annular and toroidal pseudo bracket polynomials}, which extend the Kauffman bracket polynomial from classical knot theory (\cite{Kauffman}) to annular and toroidal pseudo links. A bracket polynomial for planar pseudo links was previously defined as a 2-variable Laurent polynomial in \cite{HD}, which was then extended to a 3-variable Laurent polynomial for annular pseudo links in \cite{D}.

Here, we first introduce the notion of {\it regular isotopy} for  planar, annular and toroidal pseudo link diagrams.  Regular isotopy in these diagrammatic settings extends regular isotopy of classical knot and link diagrams, as the invariance under planar isotopy moves, R0, and the Reidemeister moves  R2 and R3 \cite{Kauffman}, by incorporating the additional pseudo Reidemeister moves PR2 and PR3 (see Figure~\ref{reid}). 
We then define the planar, (universal) annular, (universal, reduced)  toroidal pseudo bracket polynomials, as $n$-variable Laurent polynomials, with $n=3$ (generalizing the one of \cite{HD}), $n=4$ (generalizing the one of \cite{D}), $n=5$ and $n=\infty$, respectively (Definitions~\ref{pkaufb},~\ref{pkaufbst} and~\ref{kbkntt}), which we prove to be invariants of regular isotopy for the corresponding categories of pseudo links in Propositions~\ref{regular_invariance} and ~\ref{annular_regular_invariance}, and Theorem~\ref{th:regular_invariance_tor}. 
By further normalizing the regular isotopy under the Reidemeister move R1 and the pseudo Reidemeister move PR1, we obtain the corresponding Jones-type invariants for planar (recovering the one in \cite{HD}), annular and toroidal pseudo links, as well as their three-dimensional lifts  in the three-sphere, the solid torus, and the thickened torus, as stated in Theorems~\ref{th:pl-bracket}, ~\ref{th:an-bracket}, and ~\ref{th:tor-bracket}. We recapitulate our results in the following theorems.\\

\noindent {\bf Theorem 1.} 
\textit{The (universal) annular pseudo bracket polynomial, resp. the normalized annular pseudo bracket polynomial, is a regular isotopy invariant, resp. an isotopy invariant, of annular pseudo links and their corresponding pseudo links in the solid torus.}\\

\noindent {\bf Theorem 2.} 
\textit{The  (universal, reduced) toroidal pseudo bracket polynomial, resp. the normalized (universal, reduced) toroidal pseudo bracket polynomial, is a regular isotopy invariant, resp. an isotopy invariant, of toroidal pseudo links and their corresponding pseudo links in the thickened torus.}

\smallbreak
Furthermore, the definitions of pseudo links in the solid and in the thickened torus, as lifts of annular and toroidal pseudo links, also allow us to translate the theories of annular and toroidal pseudo links into the theories of ${\rm O}$-mixed pseudo links and ${\rm H}$-mixed pseudo links in the three-sphere $S^3$, respectively, cf. \cite{DLM3}. More precisely, an {\it ${\rm O}$-mixed pseudo link}  is a spatial pseudo link in $S^3$ that contains a point-wise fixed unknotted component representing the complementary solid torus (see also \cite{D}). Further, an {\it ${\rm H}$-mixed pseudo link} is a spatial pseudo link in $S^3$ that contains a point-wise fixed Hopf link as a sublink, whose complement is a thickened torus. On the diagrammatic level, these imply that annular and toroidal pseudo links can be viewed as planar ${\rm O}$-mixed and ${\rm H}$-mixed pseudo link diagrams. 

We thus introduce the 4-variable {\it mixed bracket polynomial for planar ${\rm O}$-mixed pseudo links} (Definition~\ref{mix-pkaufbst}) and its universal analogue, which we prove to be  invariants under regular isotopy of ${\rm O}$-mixed pseudo links, equivalent to the (universal)  annular pseudo bracket polynomial (Theorem~\ref{th:an-mix-in}). Furthermore, we introduce the infinite variable  {\it (universal) mixed bracket polynomial for planar ${\rm H}$-mixed pseudo links}  (Definition~\ref{mix-pkaufbtt}), and a reduced 5-variable version (Definition~\ref{def:reduced-H-bracket}), which we prove to be invariants under regular isotopy of ${\rm H}$-mixed pseudo links, equivalent to the (universal) toroidal  and the reduced toroidal pseudo bracket polynomials (Theorem~\ref{th:th-mix-in}). Finally, in analogy to the annular and toroidal pseudo links, we normalize the (universal) mixed pseudo link polynomials under the moves R1 and PR1, and thus obtain Jones-type invariants for annular and toroidal pseudo links through the theory of mixed links (Theorems ~\ref{th:an-bracket} and ~\ref{th:tor-bracket}).\\

\noindent {\bf Theorem 3.} 
\textit{The ${\rm O}$-mixed pseudo bracket polynomial, resp. the normalized ${\rm O}$-mixed pseudo bracket polynomial, is a regular isotopy invariant, resp. an isotopy invariant, of ${\rm O}$-mixed pseudo links and their corresponding spatial ${\rm O}$-mixed pseudo links.}\\

\noindent {\bf Theorem 4.} 
\textit{The ${\rm H}$-mixed pseudo bracket polynomial, resp. the normalized ${\rm H}$-mixed pseudo bracket polynomial, is a regular isotopy invariant, resp. an isotopy invariant, of ${\rm H}$-mixed pseudo links and their corresponding spatial ${\rm O}$-mixed pseudo links.}\\

The study of annular and toroidal pseudo knots is driven not only by their classification, but also by their connections to periodic pseudo tangles, which are defined as their universal covers in a thickened ribbon and in the thickened plane respectively, and their potential applications. For a detailed exploration of pseudo doubly periodic tangles, we refer the reader to \cite{DLM4}.

\smallbreak

The paper is organized as follows. In \S~\ref{sec:planar} we recall the basic notions associated with planar pseudo links, including the pseudo Reidemeister equivalence, the lift in three-dimensional space and the planar bracket and Jones polynomials, after introducing the notion of regular isotopy for pseudo links. In \S~\ref{sec:annular} we recall the theory of annular pseudo links, that is, their  pseudo Reidemeister equivalence and their lifts in the solid torus. We also recall their representation as planar ${\rm O}$-mixed pseudo links and their equivalence relation. We then extend the bracket polynomial for annular pseudo links and planar ${\rm O}$-mixed pseudo links, proving their invariance under regular isotopy, and we normalize them into Jones-type polynomials. In \S~\ref{sec:toroidal} we proceed analogously by first recalling the theory of toroidal pseudo links, including their  pseudo Reidemeister equivalence and their lifts in the thickened torus. We also recall their representation as planar ${\rm H}$-mixed pseudo links and their equivalence relation. We then define the bracket polynomials for toroidal pseudo links and planar ${\rm H}$-mixed pseudo links, and we prove their invariance under regular isotopy. Finally, we normalize these polynomials into Jones-type polynomials, defining new invariants for classifying toroidal pseudo links.

\section{Planar Pseudo Links}\label{sec:planar}

The classical by now theory of pseudo knots and links is a diagrammatic theory introduced by Hanaki in \cite{H} and the mathematical background of the theory was established in \cite{HJMR}. In this section we start with a summary of the main concepts of planar pseudo links, as presented in \cite{DLM3}, including their lifts in three-space and their equivalence relation through the generalized Reidemeister theorem. We then recall key results on well-known polynomial invariants of planar pseudo links, namely the planar pseudo bracket and Jones polynomials, which will serve as a basis for the study of annular and toroidal pseudo links in the next sections.

\subsection{Preliminaries}

A {\it planar pseudo link diagram} is defined as a regular link diagram in the plane containing a number of classical crossings as well as crossings in which the crossing information is missing, that is, it is unknown which arc passes over and which arc passes under the other.  These undetermined crossings, called {\it precrossings} or {\it pseudo crossings}, are marked by transversal intersections of arcs of the diagram enclosed in a light gray circle, as illustrated in the example in Figure~\ref{pk1}. Assigning an orientation to each component of a pseudo link diagram results in an {\it oriented} pseudo link diagram. 

\begin{figure}[H]
\begin{center}
\includegraphics[width=0.8in]{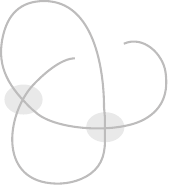}
\end{center}
\caption{A pseudo knot.}
\label{pk1}
\end{figure}

We consider planar pseudo link diagrams under the discrete equivalence relation, called {\it isotopy}, induced by planar isotopy moves, denoted R0, the classical Reidemeister moves R1, R2, R3 and the {\it pseudo Reidemeister moves} PR1, PR2, PR3, as exemplified in Figure~\ref{reid}, with all their relevant variants. The equivalence class of a planar pseudo link diagram is a {\it planar pseudo link}. 

We also consider planar pseudo link diagrams under a more restricted equivalence relation, which extends the equivalence of regular isotopy introduced by Kauffman \cite{Kauffman}, as follows:

\begin{definition}\label{def:pseudo-regular}
{\it Regular isotopy for pseudo links} is the equivalence relation among pseudo link diagrams, induced by the moves R0, R2, R3, PR2, and PR3. 
\end{definition}

When considering either equivalence class of an oriented pseudo link diagram, the oriented versions of the respective moves are needed to preserve the orientation. 

\begin{figure}[H]
\begin{center}
\includegraphics[width=4.7in]{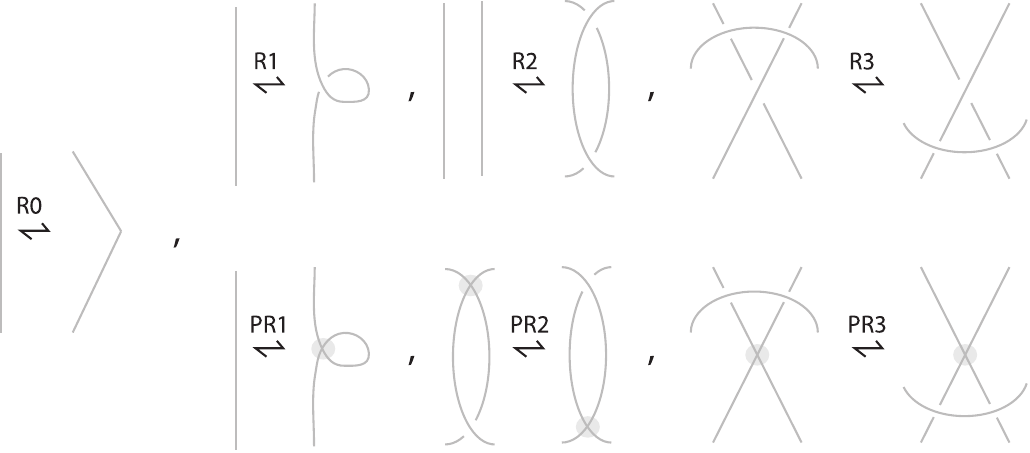}
\end{center}
\caption{Reidemeister equivalence moves for pseudo link diagrams.}
\label{reid}
\end{figure}

\begin{remark}
If $K$ is a pseudo link diagram without precrossings, then $K$ can be viewed as a classical link diagram and its equivalence class under the classical moves R0, R1, R2, R3 is preserved by the above pseudo isotopy. In this sense, there is an injection of the classical link types into the pseudo link types. 
\end{remark}

We now recall from \cite{DLM3} the definition of the lift of a pseudo link diagram. 

\begin{definition} \label{spatiallift} 
The {\it lift} of a planar pseudo link diagram in the three-dimensional space is a collection of closed curve(s) defined such that every classical crossing is embedded in a sufficiently small 3-ball, while every precrossing is supported by a sufficiently small rigid disc embedded in three-space. The simple arcs connecting crossings can also be replaced by isotopic ones embedded in three-space. We call the resulting lift a {\it spatial pseudo link}.
\end{definition}

\begin{figure}[H]
\begin{center}
\includegraphics[width=2.3in]{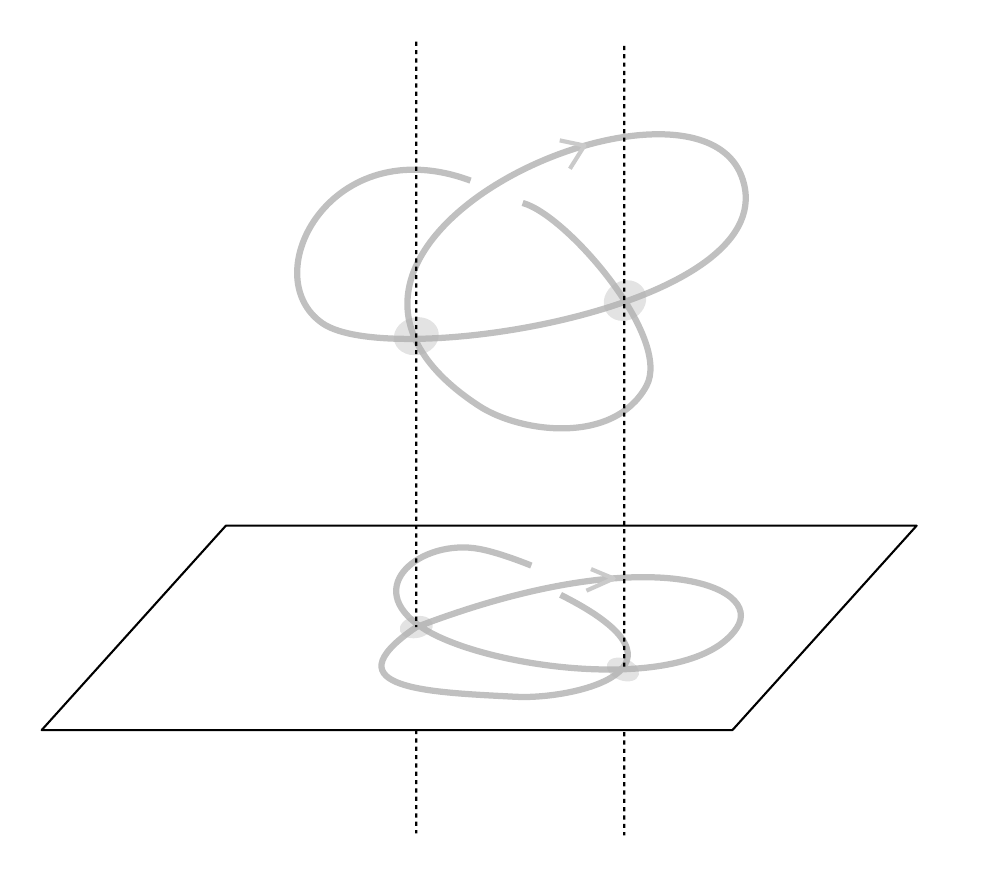}
\end{center}
\caption{The lift of a pseudo knot diagram to a spatial pseudo knot (cf. \cite{DLM3}).}
\label{trefoil-lifting}
\end{figure}

The inverse operation of the lift is the regular projection of a spatial pseudo link on a specified plane, whereby a disc supporting a pseudo crossing does not project on an arc, and this is a planar pseudo link diagram. 

Recall now that two (oriented) spatial pseudo links are said to be {\it (ambient) isotopic} if they are related by arc and disc isotopies (cf. \cite{DLM3}). Then, a generalization of the Reidemeister theorem for spatial pseudo links is stated as follows. {\it Two (oriented) spatial pseudo links are (ambient) isotopic if and only if any two corresponding planar pseudo link diagrams of theirs are (oriented) pseudo isotopic.}

\subsection{The classical bracket polynomial}

The bracket polynomial, denoted by $\langle \cdot \rangle$, was introduced by Kauffman in \cite{Kauffman} as a one-variable Laurent polynomial in the ring $ \mathbb{Z}[A^{\pm 1}]$ defined recursively by means of the following rules: 

\begin{enumerate} 
  \item $\langle {\rm O} \rangle = 1$, with O denoting the standard unknot
  \smallbreak
  \item $\langle D \cup {\rm O} \rangle = (- A^2 - A^{-2}) \langle D \rangle$
  \smallbreak
  \item $\langle$ \raisebox{-3pt}{\includegraphics[scale=0.45]{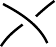}} $\rangle$ = $A$ $\langle$  \raisebox{-0pt}{\includegraphics[scale=0.45]{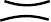}} $\rangle$ $+$ $A^{-1}$ $\langle$  \raisebox{-2pt}{\includegraphics[scale=0.6]{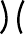}} $\rangle$, for diagrams that differ locally around a single crossing.  
\end{enumerate}

\smallbreak
The bracket polynomial is an invariant of regular isotopy for classical (unoriented) links, meaning that it remains unchanged under equivalence induced by the R0, R2 and R3 moves. To obtain an ambient isotopy invariant for classical knots and links, one needs to normalize the bracket polynomial so as to also respect move R1. This normalized version of the bracket polynomial is equivalent to the {\it Jones polynomial} of a link. The normalization process involves assigning an orientation to each component of a link $L$, which in turn assigns $+1$ or $-1$ to each crossing of a given diagram, say $D$, of $L$, as shown in the convention in Figure~\ref{sign}.

\begin{figure}[H]
\begin{center}
\includegraphics[width=1.4in]{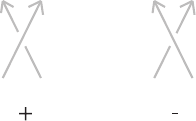}
\end{center}
\vspace{8pt}
\caption{The signs of the crossings.}
\label{sign}
\end{figure} 

Further, the {\it writhe} of the diagram $D$ is defined as the sum of the signs of the crossings in the diagram. More formally, if $C$ is the set of crossings in $D$, the writhe $w(D)$ is given by:

\begin{equation}\label{writhe}
w(D) = \sum_{c \in C} \text{sgn}(c),
\end{equation}

\noindent where $\text{sgn}(c)$ denotes the sign of crossing $c$. Note that the writhe itself is an invariant of regular isotopy. Then, the polynomial $V_L \in \mathbb{Z}[A^{\pm 1}]$ is defined via the bracket polynomial $\langle L \rangle$ of the link $L$, as follows, and it is an ambient isotopy invariant of $L$ (cf. \cite{Kauffman}):

\[
V_L(A) = (-A)^{-3 w(D)} \langle L \rangle.
\]

\noindent Further, applying the variable substitution $A=t^{-1/4}$ in the polynomial $V_L(A)$ one obtains the well-known Jones polynomial.

\subsection{The planar pseudo bracket polynomial}

In \cite{HD} the 2-variable pseudo bracket polynomial is defined for planar pseudo link diagrams, extending the Kauffman bracket polynomial $\langle \cdot \rangle$  for classical knots and links \cite{Kauffman}. This extension was achieved by considering {\it oriented} pseudo link diagrams, which allowed the definition of a skein relation at a precrossing. Note that this skein relation can be viewed as an operation where specific tangles are inserted to replace the precrossings. See \cite{AHLK} for the application of tangle insertion to pseudo knots. In the setting of \cite{HD}, which aimed at an ambient isotopy invariant, the pseudo bracket polynomial was forced to satisfy the move PR1. However, the move PR1 is not included in our definition of regular isotopy for pseudo links (Definition~\ref{def:pseudo-regular}), so we define the pseudo bracket polynomial  as a 3-variable polynomial as follows:

\begin{definition}\label{pkaufb}\rm
Let $K$ be an oriented planar pseudo link diagram. The {\it planar pseudo bracket polynomial} of $K$, denoted $\langle K \rangle$, is a Laurent polynomial in the ring $\mathbb{Z}\left[A^{\pm 1}, V, H \right]$ that is defined  by means of the following recursive rules:
\begin{figure}[H]
\begin{center}
\includegraphics[width=3.3in]{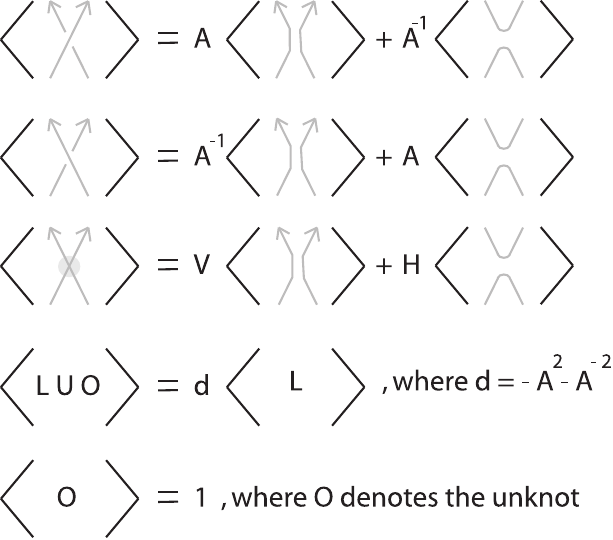}
\end{center}
\label{pkb}
\end{figure}
\end{definition}

\begin{remark} 
Note that the first two rules can be replaced by the well-known one rule of the classical bracket polynomial with unoriented arcs (rule (3) above). It follows that, in the absence of precrossings, the pseudo bracket polynomial coincides with the classical Kauffman bracket polynomial. 
\end{remark}

By the same arguments as in \cite{HD}, $\langle K \rangle$ is invariant under the moves R0, R2, R3, PR1, PR2, and PR3. Hence, it is also invariant under the subset of moves R0, R2, R3, PR2, and PR3, which according to Definition~\ref{def:pseudo-regular} generate pseudo regular isotopy. Thus, we can state that:

\begin{proposition} \label{regular_invariance}
The pseudo bracket polynomial is invariant under regular isotopy of oriented planar pseudo links. 
\end{proposition}

\begin{example} \label{exampleptrefoil}
We compute the pseudo bracket polynomial of the pseudo trefoil knot $K$ of Figure~\ref{pk1}, whose skein tree is depicted in Figure~\ref{pseudo-spherical} and we obtain the following: 
$$\langle K \rangle = - V^2 A^3 - VHA^{-3} + VHA^{-1} - H^2 A^{-3} + VHA^{-5}$$

\begin{figure}[H]
\begin{center}
\includegraphics[width=6in]{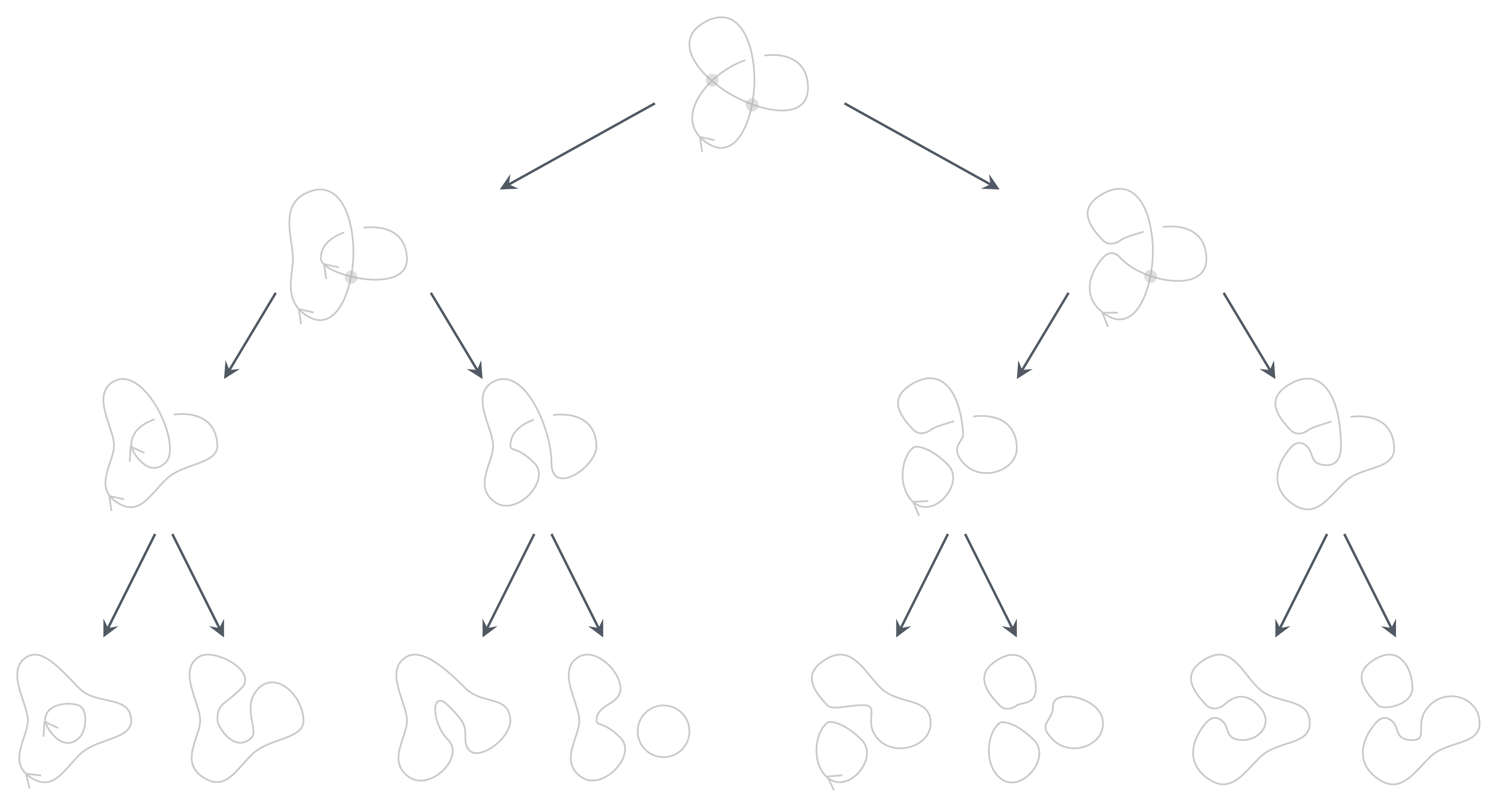}
\end{center}
\caption{The skein tree of a pseudo trefoil knot.}
\label{pseudo-spherical}
\end{figure}
\end{example}

\subsection {The planar pseudo Jones polynomial}\label{sec:planar-Jones} 

As in the case of the classical bracket polynomial \cite{Kauffman} for links in $S^3$, one can normalize the pseudo bracket polynomial in order to satisfy the moves R1 and PR1. For the invariance under the move PR1, the specialization $H = 1-Vd$ must be forced, as shown in \cite{HD}. For the invariance under the move R1, we  consider the product of $\langle L \rangle$ by the factor $\left( -A^{-3}\right) ^{w(L)}$ as in the normalization of the classical bracket polynomial. Note that we retain the definition of writhe $w(L)$ of a pseudo link diagram $L$, as defined in \cite{HD}, where precrossings do not contribute to the counting of signs. So the writhe $w(L)$ is defined as in Equation~\ref{writhe} and it is  an invariant of regular isotopy for pseudo links. Hence we have the following (see \cite[Corollary~2]{HD}): 

\begin{theorem}\label{th:pl-bracket}
Let $L$ be an oriented planar pseudo link diagram. The normalized pseudo bracket polynomial is a 2-variable Laurent polynomial in the ring $\mathbb{Z}\left[A^{\pm 1}, V\right]$, defined as
\[
P_L(A, V)\ =\ (-A^{-3})^{w(L)}\, \langle L \rangle,
\]
\noindent where $w(L):=\underset{c\in C(L)}{\sum}\, sgn(c)$, $C(L)$ the set of classical crossings of $L$ and $\langle L \rangle$ denotes the pseudo bracket polynomial of $L$ for  $H = 1-Vd$, and it is an isotopy invariant of $L$.
\end{theorem}

The normalized pseudo bracket polynomial is analogous to the normalized bracket for classical links and it can be considered as an equivalent to the Jones polynomial for planar/spatial pseudo links:  

\begin{definition}
Applying the variable substitution $A=t^{-1/4}$ in the polynomial $P_L(A, V)$ we obtain the equivalent  {\it pseudo Jones polynomial}, which is an invariant of planar pseudo links and their corresponding spatial pseudo links.
\end{definition}

\begin{remarks}\rm
Consider the pseudo trefoil of Example~\ref{exampleptrefoil}. We observe that:
\begin{itemize}
    \item[i.] Its pseudo bracket polynomial contains the new variable $V$, which witnesses the presence of pseudo crossings (in fact $V^2$ for two pseudo crossings), and this will remain after applying the writhe normalization factor. So, the polynomial $P$ distinguishes the pseudo trefoil from the corresponding classical trefoil or the unknot. 
     \item[ii.] In the absence of precrossings, $P_L(A, V)$ resp.  $P_L(t^{-1/4}, V)$ specializes to the normalized  bracket resp. the Jones polynomial for  classical links.
\end{itemize}
\end{remarks}

\section{Annular pseudo links}\label{sec:annular}

In this section, we recall essential results on annular pseudo links as presented in \cite{DLM3}. In particular we recall the bijections between annular pseudo links and their lift in the solid torus as well as with the  ${\rm O}$-mixed pseudo links in the three-sphere. We then define the bracket polynomial for annular pseudo links (cf. also \cite{D}) and for ${\rm O}$-mixed pseudo links. We finally normalize to the corresponding Jones polynomials, which we prove to be topological invariants of  annular pseudo links, pseudo links in the solid torus and ${\rm O}$-mixed pseudo links in the three-sphere, respectively.

\subsection{Preliminaries}\label{sec:prel-annular}

An {\it annular pseudo link diagram} is defined as a pseudo link diagram in the annulus $\mathcal{A}$. For an illustration see Figure~\ref{pkannulus}. An \textit{annular pseudo link} is an equivalence class of annular pseudo link diagrams related by a finite sequence of surface isotopy moves in $\mathcal{A}$, the classical Reidemeister moves and the pseudo Reidemeister moves (recall Figure~\ref{reid}). This equivalence relation between annular pseudo link diagrams is called {\it isotopy}. Further, according to Definition~\ref{def:pseudo-regular}, {\it regular isotopy } among annular pseudo link diagrams is induced by all moves above with the exception of the moves R1 and PR1. 

\begin{figure}[H]
\begin{center}
\includegraphics[width=2in]{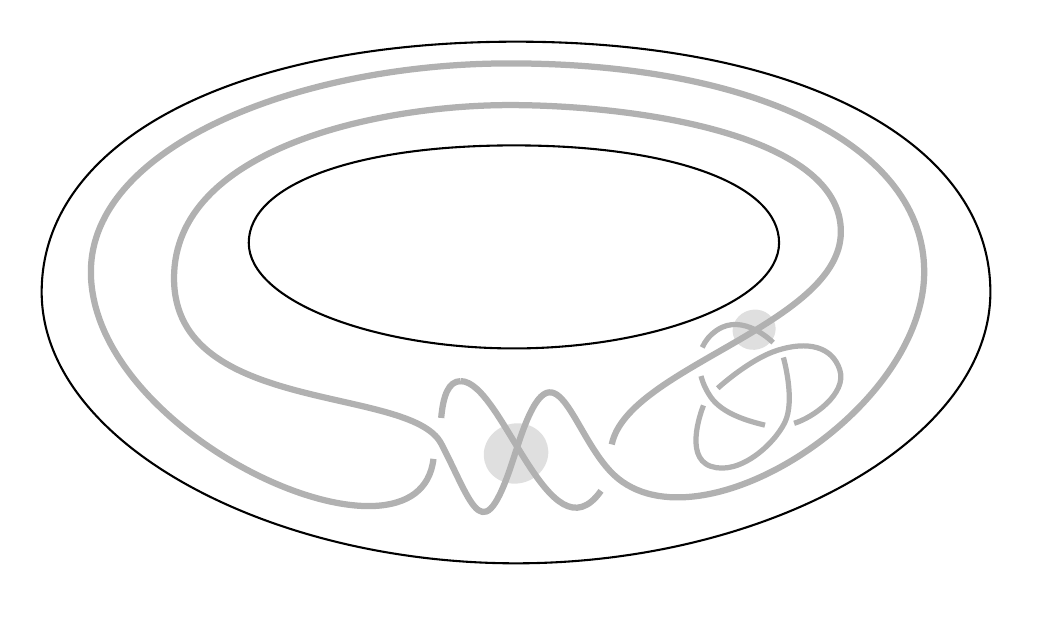}
\end{center}
\caption{An annular pseudo link.}
\label{pkannulus}
\end{figure}

Annular pseudo links can be studied in different topological contexts. Recall first that the {\it lift} of an annular pseudo link diagram in the thickened annulus $\mathcal{A} \times I$, where $I$ denotes the unit interval $[0,1]$, is defined as a collection of closed curve(s) constrained to the interior of the solid torus ST - which is homeomorphic to $\mathcal{A} \times I$ - consisting of embedded discs from which emanate embedded arcs \cite{DLM3}. More specifically, each classical crossing of the diagram is embedded in a sufficiently small 3-ball that lies entirely within $\mathcal{A} \times I$, while precrossings are supported by sufficiently small rigid discs, which are embedded in $\mathcal{A} \times I$. Note also that the simple arcs connecting the (pre)crossings can be replaced by isotopic ones in $\mathcal{A} \times I$. The resulting lift is called a {\it pseudo link in the solid torus}. For an illustration of the lift of an annular pseudo link diagram to a pseudo link in the solid torus, we refer to Figure~\ref{annularlifting}. 

\begin{figure}[H]
\begin{center}
\includegraphics[width=2.3in]{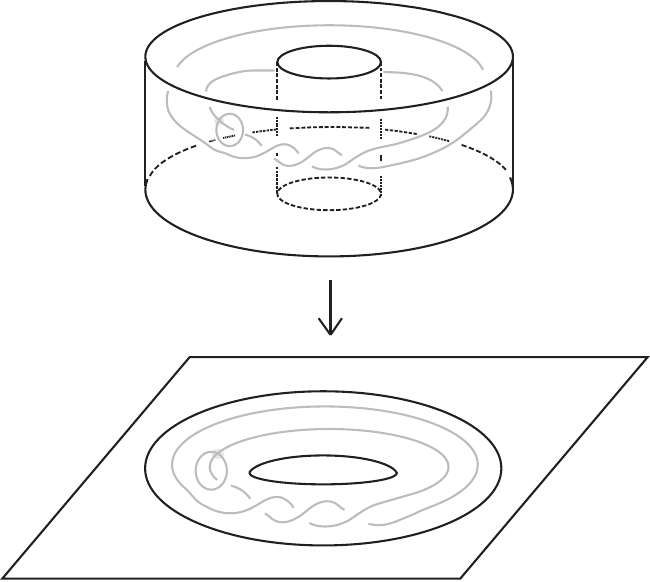}
\end{center}
\caption{The lift of an annular pseudo link to a  pseudo link in $\mathcal{A} \times I$ (cf. \cite{DLM3}).}
\label{annularlifting}
\end{figure}

Two (oriented) pseudo links in the solid torus are said to be {\it (ambient) isotopic} if they are related by isotopies of arcs and discs that are confined to the interior of the solid torus. In \cite{DLM3} we proved the following theorem:

\begin{theorem}\label{isopST}
Two (oriented) pseudo links in the solid torus $\mathcal{A} \times I$ are (ambient) isotopic if and only if any two corresponding annular pseudo link diagrams of theirs, projected onto the annulus $\mathcal{A} \times \{0\}$, are related by surface isotopies, classical Reidemeister moves and the pseudo Reidemeister moves.
\end{theorem}

Another approach to the theory of annular pseudo links is to represent annular pseudo link diagrams, resp.  pseudo links in the solid torus, as mixed pseudo link diagrams in the plane, resp. as mixed pseudo links in the three-sphere $S^3$. To be more specific, we recall from \cite{LR1} that isotopy classes of links in a knot/link complement correspond bijectively to isotopy classes of mixed links in $S^3$, through isotopies that preserve a fixed sublink, which represents the knot/link complement, and act only on the {\it moving part}, which represents the link in the knot/link complement.  In our case, viewing the solid torus ST as the complement of a solid torus in $S^3$, a link in ST is represented uniquely by a mixed link in $S^3$ that contains the standard unknot, ${\rm O}$, as a point-wise fixed sublink representing the complementary solid torus. 

In analogy, in \cite{DLM3}, and using the notion of  lift of an annular pseudo link diagram,  we define the notion of an (oriented) ${\rm O}$-\textit{mixed pseudo link} in $S^3$ as a mixed link containing the standard unknot ${\rm O}$ as the fixed part and the lift of an annular pseudo link diagram as the moving part, with no precrossings between the fixed and the moving part. Clearly, an (oriented) ${\rm O}$-\textit{mixed pseudo link diagram} is a regular projection, on the plane of ${\rm O}$, of an (oriented) ${\rm O}$-mixed pseudo link ${\rm O}\cup K$, such that some double points are precrossings, as projections of the precrossings of the moving part $K$,  and the rest of the crossings are either crossings of arcs of the moving part or \textit{mixed crossings} between arcs of the moving and the fixed part, endowed with over/under information.  

Then, {\it isotopy classes of (oriented) pseudo links in the solid torus are in bijective correspondence with isotopy classes of (oriented) ${\rm O}$-mixed pseudo links in $S^{3}$ via isotopies that keep ${\rm O}$ fixed.} For an example of a pseudo link in the solid torus and its transition to its corresponding ${\rm O}$-mixed pseudo link see Figure~\ref{pmixedtor}. The reader may compare with \cite{D}, where a pseudo link in ST is defined as an ${\rm O}$-mixed pseudo link.

\begin{figure}[H]
\begin{center}
\includegraphics[width=4.5in]{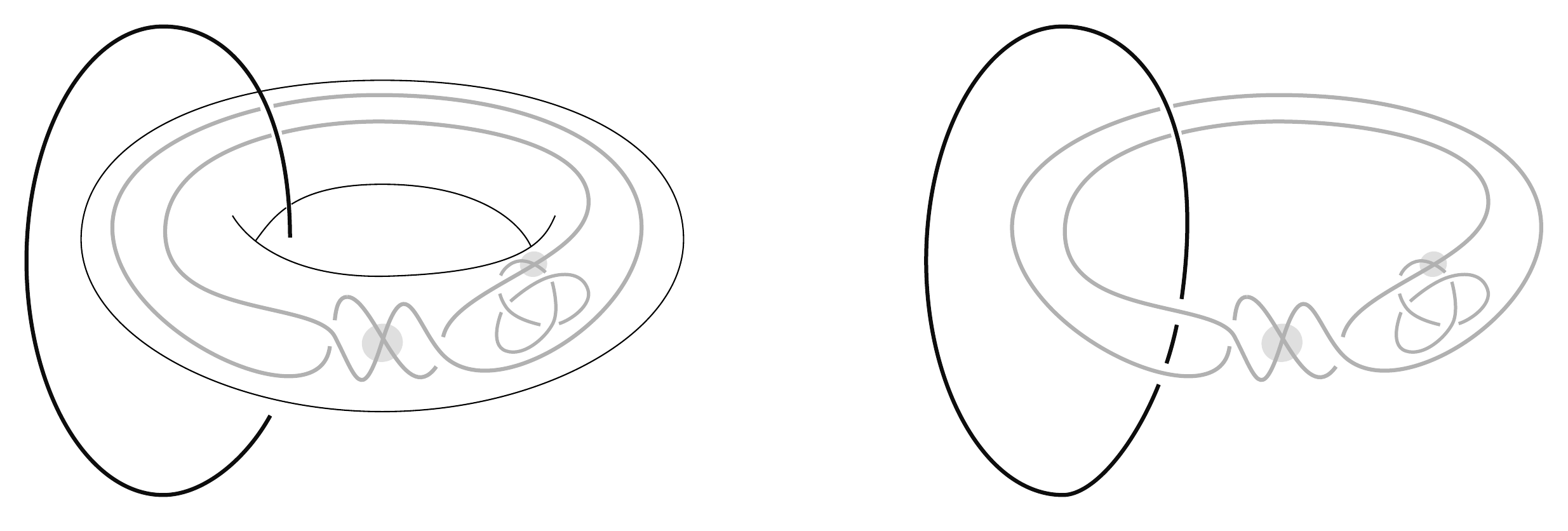}
\end{center}
\caption{A pseudo link in the solid torus and its corresponding ${\rm O}$-mixed pseudo link.}
\label{pmixedtor}
\end{figure}

The isotopy of ${\rm O}$-mixed pseudo links is translated on the diagrammatic level as follows: 

\begin{theorem} \label{Omixedreid}\cite{DLM3}
Two (oriented) ${\rm O}$-mixed pseudo links in $S^{3}$ are isotopic if and only if any two (oriented)  ${\rm O}$-mixed pseudo link diagrams of theirs  differ by planar isotopies, a finite sequence of the classical and the pseudo Reidemeister moves for the moving parts (recall Figure~\ref{reid}), and moves that involve the fixed and the moving parts, called {\rm mixed Reidemeister moves}, comprising the moves MR2, MR3, and MPR3, as exemplified in Figure~\ref{mpr}. 
\end{theorem}

\begin{figure}[H]
\begin{center}
\includegraphics[width=4.7in]{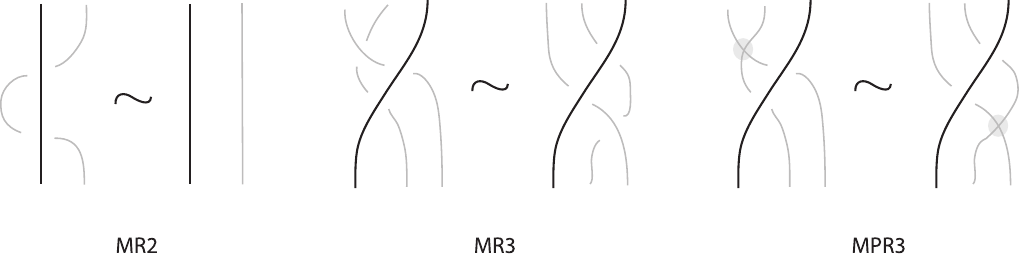}
\end{center}
\caption{The mixed Reidemeister moves for ${\rm O}$-mixed pseudo link diagrams.}
\label{mpr}
\end{figure}

It follows from the above that, {\it two (oriented) annular pseudo links diagrams are isotopic if and only if any two corresponding (oriented) ${\rm O}$-mixed pseudo link diagrams are isotopic.}

\subsection{The annular pseudo bracket polynomial}

The pseudo bracket polynomial for planar pseudo link diagrams as defined in \cite{HD} was naturally extended in \cite{D} to pseudo links in the solid torus, exploiting the theory of mixed links and representing the solid torus by the once punctured disc, and further extended in \cite{D1} to pseudo links in the handlebody of genus $g \geq 2$.

Here we extend the 3-variable bracket polynomial for planar pseudo link diagrams as defined in Definition~\ref{pkaufb} to the setting of annular pseudo links. The main difference from the planar pseudo bracket lies in the fact that when applying the skein relation, some states may contain one or more essential unknots, to which we assign a specific variable, denoted by $s$. So, any number of essential unknots along with the standard null-homotopic unknot constitute distinct initial conditions for the polynomial, which we define as follows: 

\begin{definition}\label{pkaufbst}\rm
Let $K$ be an oriented annular pseudo link diagram. The {\it annular pseudo bracket polynomial} of $K$, denoted $\langle K \rangle_\mathcal{A}$ or simply $\langle K \rangle$, is a 4-variable Laurent polynomial  in $\mathbb{Z}\left[A^{\pm 1}, V, H, s \right]$,  defined by means of the following rules, where a punctured plane represents the annulus:
\begin{figure}[H]
\begin{center}
\includegraphics[width=3.6in]{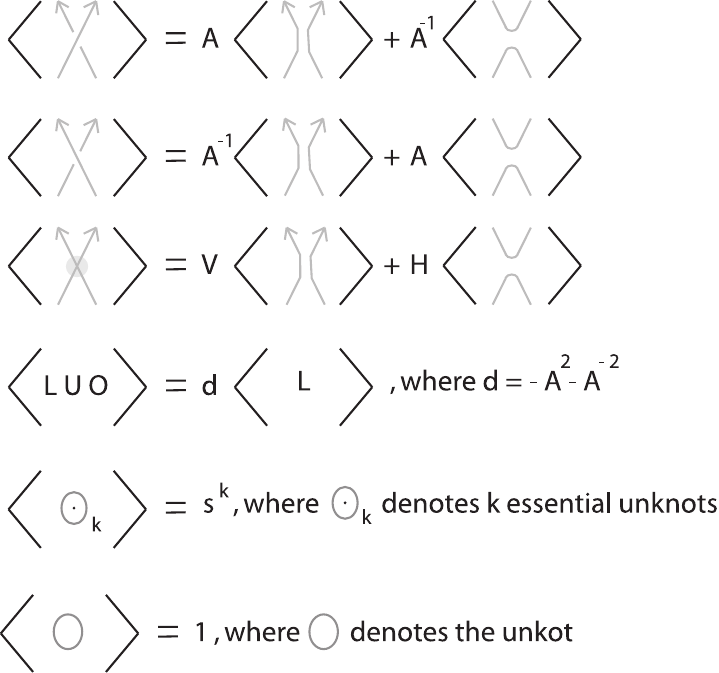}
\end{center}
\label{pkbst}
\end{figure}
\end{definition}

\noindent Note that the last rule can only be applied to an annular state diagram that contains only a standard null-homotopic unknot.

\begin{remarks} 
\,\\
\noindent (a) If we have an annular pseduo diagram that contains no essential closed curves, by using the inclusion of a disc in $\mathcal{A}$ (see left-hand illustration in Figure~\ref{planarannular}), such annular  diagrams can be viewed as planar state diagrams and resolve into states as in the planar case.
\smallbreak
\noindent  (b) In contrast, using the inclusion of $\mathcal{A}$ in a disc  (see right-hand illustration in Figure~\ref{planarannular}, cf. \cite{DLM3}), essential unknots become null homotopic, so substituting $s^k$ by $d^{k-1} = (-A^2-A^{-2})^{k-1}$, the annular pseudo bracket polynomial specializes to the planar pseudo bracket of Definition~\ref{pkaufb} for planar pseudo links.
\smallbreak
\noindent  (c) The main difference of our annular pseudo bracket polynomial from the one defined in \cite{D} is that the variable $H$ is specialized to $H=1-Vd$ in \cite{D}.   
\smallbreak
\noindent  (d) In the fifth rule of Definition~\ref{pkaufbst} the initial condition could be substituted by $s_k$ in place of $s^k$, for $k\in \mathbb{N}$, giving rise to the {\it universal annular pseudo bracket polynomial} with infinitely many variables. Since precrossings are also smoothened, the universal annular pseudo bracket polynomial recovers the Kauffman bracket skein module of annular pseudo links or pseudo links in the solid torus, cf. \cites{HP,D}. 
\end{remarks}

\begin{figure}[H] 
\begin{center} 
\includegraphics[width=5.5in]{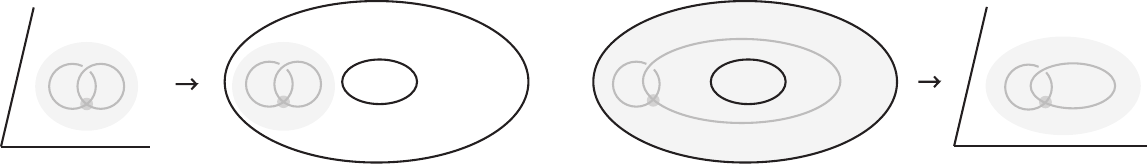} 
\end{center} 
\vspace{8pt}
\caption{Inclusion relations: a disc in the annulus and the annulus in a disc.} 
\label{planarannular} 
\end{figure} 

In analogy to the planar pseudo bracket polynomial of Definition~\ref{pkaufb}, we have the following:

\begin{proposition} \label{annular_regular_invariance}
The (universal) annular pseudo bracket polynomial is invariant under regular isotopy of oriented annular pseudo links. 
\end{proposition}

\begin{proof}
We refer to \cite[Proposition~5.5]{D} for details on the invariance under the Reidemeister moves R2, R3, and pseudo Reidemeister moves PR2, PR3, noting that in all proofs there was no need of specializing the variable $H$.
\end{proof}

\begin{example} \label{exampleannularptrefoil}
We compute the annular pseudo bracket polynomial of an annular pseudo trefoil knot $K$, whose skein tree is analyzed in Figure~\ref{pseudo-annular}. We have the following:

\[
\langle K \rangle_\mathcal{A} = - V^2 s^2 A^3 - VHA^{-3} + 2VHA^{-1} + V^2 A^{-1} - V^2 s^2 A^{-1} - H^2 A^{-3} + VHA^{-5}.
\]

\begin{figure}[H]
\begin{center}
\includegraphics[width=6in]{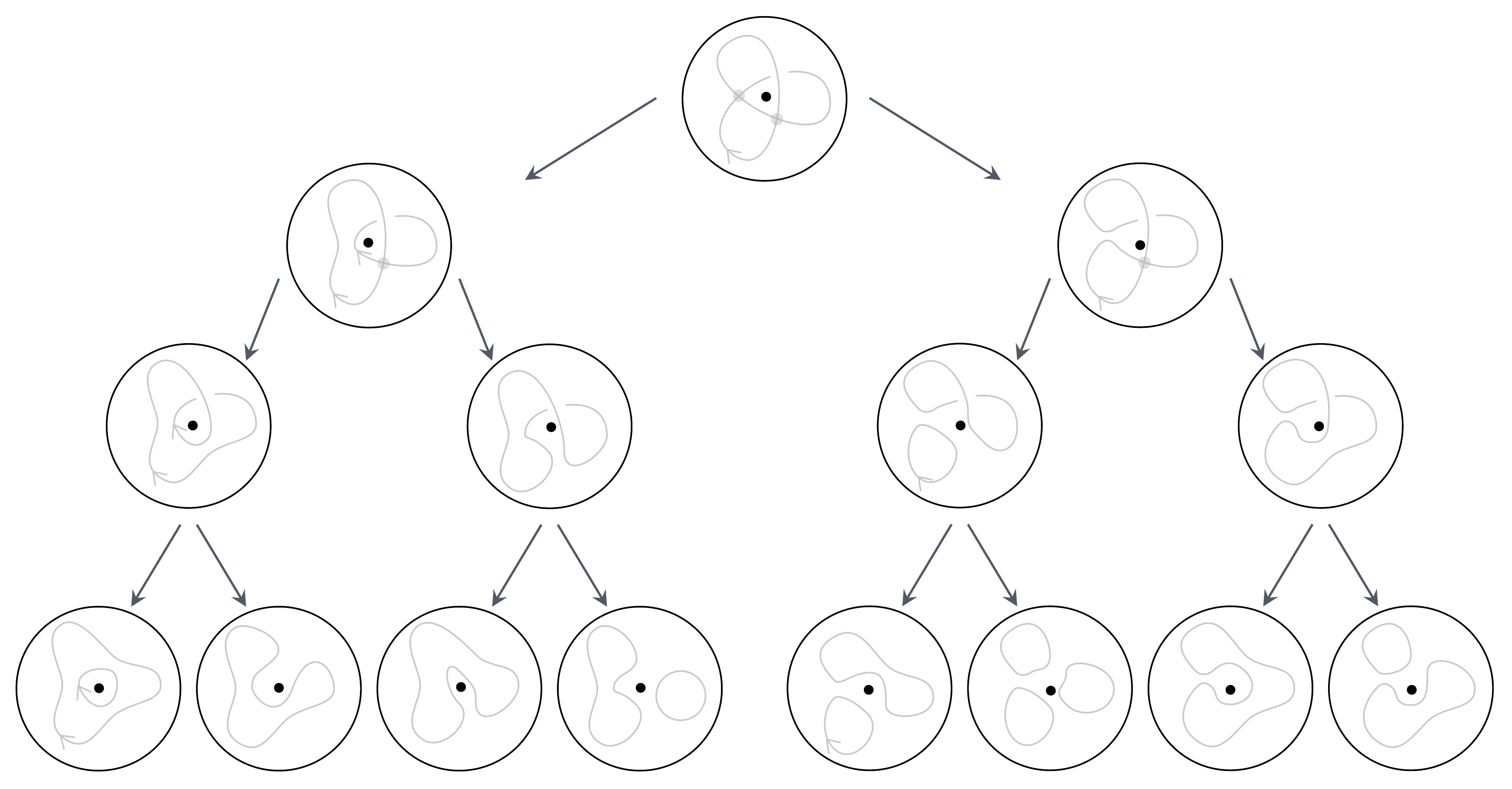}
\end{center}
\vspace{8pt}
\caption{The skein tree of an annular pseudo trefoil.}
\label{pseudo-annular}
\end{figure}

\end{example}

\subsection{The ${\rm O}$-mixed pseudo bracket polynomial}

In this subsection, we consider the planar ${\rm O}$-mixed pseudo link approach to annular pseudo links (recall Subsection~\ref{sec:prel-annular}). In view of Theorem~\ref{Omixedreid}, the annular pseudo bracket polynomial can be adapted to the setting of (oriented) ${\rm O}$-mixed pseudo links. Here the mixed crossings are not subjected to the inductive rules of Definition~\ref{pkaufbst}, since ${\rm O}$ represents the manifold, so must remain fixed throughout. Further, an essential closed curve embedded in the annulus is represented by a mixed Hopf link, as illustrated in Figure~\ref{ppsttt} (cf. \cite{La2}, see also \cite{DLM3}). 

\begin{figure}[H]
\begin{center}
\includegraphics[width=4.5in]{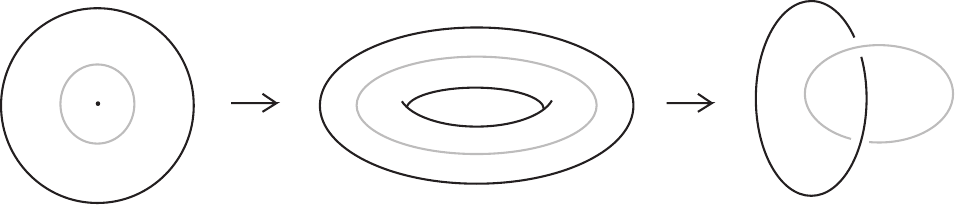}
\end{center}
\vspace{8pt}
\caption{The transition of an essential unknotted curve in $\mathcal{A}$ to an essential unknotted curve in the solid torus, to an ${\rm O}$-mixed pseudo link.}
\label{ppsttt}
\end{figure}

In the case of ${\rm O}$-mixed pseudo links it may look as though a generic state diagram includes finitely many essential unknotted components winding around the fixed ${\rm O}$-curve  some number of times, as illustrated in Figure~\ref{fig:state-O-mixed}. However, in the Kauffman bracket skein module of the solid torus such diagrams reduce to diagrams with essential unknotted components winding once around the fixed ${\rm O}$-curve, as illustrated in Figure~\ref{fig:enter-label}. Cf. \cites{D3,Bo}.

\begin{figure}[H]
    \centering
    \includegraphics[width=0.4\linewidth]{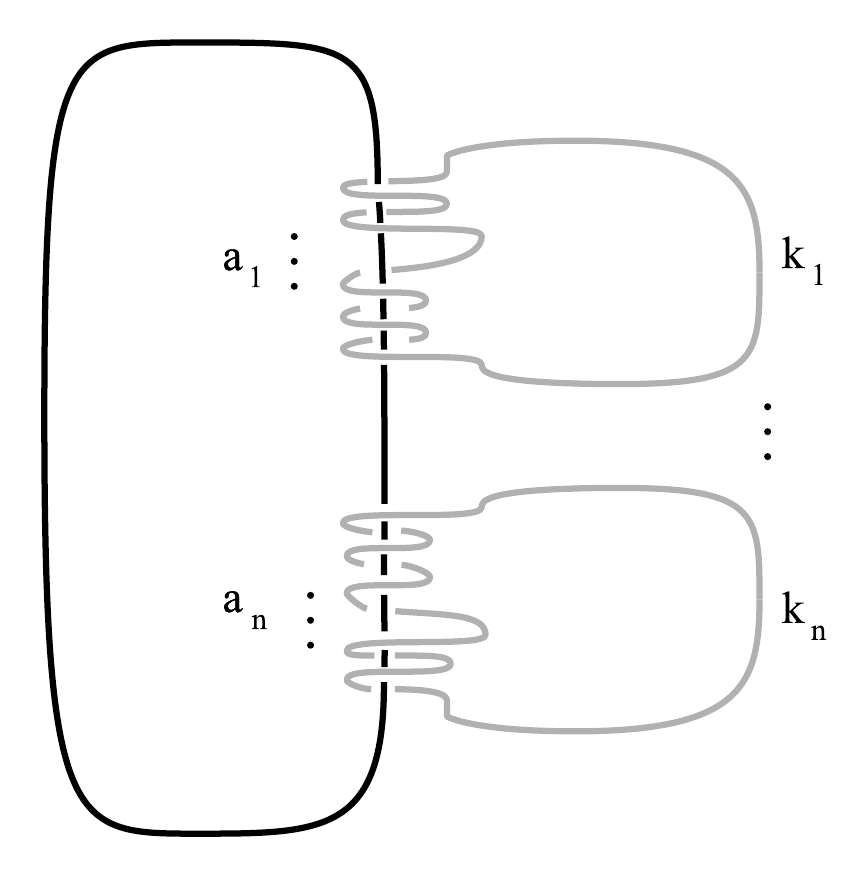}
    \caption{A state diagram with $k_1, \ldots, k_n$ identical essential unknots, winding $a_1, \ldots, a_n$ times around the ${\rm O}$-curve, respectively.}
    \label{fig:state-O-mixed}
\end{figure}

We may now define:

\begin{definition}\label{mix-pkaufbst}\rm
Let ${\rm O}\cup K$ be an oriented ${\rm O}$-mixed pseudo link diagram. The \textit{${\rm O}$-mixed pseudo bracket polynomial} of ${\rm O}\cup K$ is a 4-variable Laurent polynomial in $\mathbb{Z}\left[A^{\pm 1}, V, H, s \right]$,  defined by means of the same inductive rules as the ones for the annular mixed pseudo bracket (Definition~\ref{pkaufbst}), except for the diagram in the fifth rule in Figure~\ref{pkbst} which is substituted by the planar diagram in Figure~\ref{fig:enter-label}, where the $k$-component unlink is linked with the fixed unknot~${\rm O}$. Also, $L$ in the fourth rule stands now for an  oriented ${\rm O}$-mixed pseudo link diagram. 
\end{definition}

Further, in analogy to the universal  annular pseudo bracket polynomial  we have the  \textit{universal ${\rm O}$-mixed pseudo bracket polynomial} with infinitely many variables $s_k$, for $k\in \mathbb{N}$, in place of $s$.

\begin{figure}[H]
    \centering
    \includegraphics[width=0.3\linewidth]{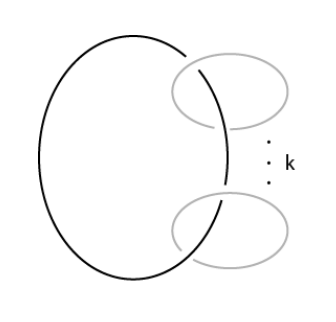}
    \caption{$k$ essential unknotted curves in the ${\rm O}$-mixed pseudo link setting.}
    \label{fig:enter-label}
\end{figure}

We then have for planar ${\rm O}$-mixed pseudo links:

\begin{theorem}\label{th:an-mix-in}
The (universal) ${\rm O}$-mixed pseudo bracket is invariant under regular isotopy of (oriented) ${\rm O}$-mixed pseudo links and it is equivalent to the (universal) annular pseudo bracket polynomial for annular pseudo links. 
\end{theorem}

\begin{proof}
We consider the restriction of the theory of planar pseudo links to the subset of ${\rm O}$-mixed pseudo links. Then, by Proposition~\ref{regular_invariance}, the (oriented) ${\rm O}$-mixed pseudo bracket polynomial is invariant under the classical Reidemeister moves R2 and R3 and under the pseudo Reidemeister moves PR2 and PR3. Regarding now the mixed Reidemeister moves (recall Figure~\ref{mpr}), we firstly observe that: as there is no rule for smoothing mixed crossings, the moves MR2 is undetectable by the ${\rm O}$-mixed pseudo bracket, so it remains invariant. Invariance under the moves MR3 and MPR3 follows immediately by applying the respective skein rule for smoothing the (pre)crossing and employing the moves MR2, as illustrated in Figure~\ref{kf1} for either type of crossings and with the coefficients omitted. 

Finally, the equivalence of the ${\rm O}$-mixed pseudo bracket polynomial for  ${\rm O}$-mixed pseudo links to the annular pseudo bracket polynomial for annular pseudo links follows immediately by Theorem~\ref{Omixedreid}.
\end{proof}

\begin{figure}[H]
\begin{center}
\includegraphics[width=4.5in]{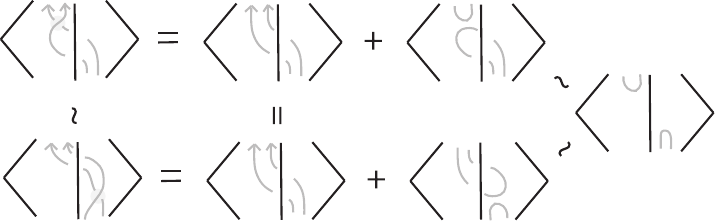}
\end{center}
\vspace{8pt}
\caption{The invariance of the ${\rm O}$-mixed pseudo bracket under the mixed Reidemeister moves MR3 and MPR3.}
\label{kf1}
\end{figure}

\subsection{Jones-type analogues}

Similarly to the case of planar pseudo links, we may normalize the annular and the ${\rm O}$-mixed pseudo bracket polynomials in order to satisfy the Reidemeister moves R1 and PR1 and obtain invariants for annular pseudo links and their corresponding pseudo links in the solid torus resp. for ${\rm O}$-mixed pseudo links. For this, we first define in analogy to the writhe of a planar pseudo link diagram:

\begin{definition} \label{def:annular_writhe}
The \textit{writhe} of an oriented annular pseudo link diagram $L$, denoted as $w(L)$, is defined as the number of positive crossings minus the number of negative crossings of $L$ (recall Figure~\ref{sign}), while the precrossings do not contribute to the writhe. Furthermore, the ${\rm O}$-\textit{writhe} of an oriented ${\rm O}$-mixed pseudo link diagram ${\rm O}\cup L$, denoted $w({\rm O}\cup L)$, is defined as the number of positive crossings minus the number of negative crossings of $L$ (recall Figure~\ref{sign}, while the precrossings and the mixed crossings do not contribute to the ${\rm O}$-writhe. 
\end{definition}

We now show that the annular bracket polynomial is an invariant of annular pseudo links and pseudo links in the solid torus, and that its ${\rm O}$-mixed bracket polynomial counterpart, is an invariant of ${\rm O}$-mixed pseudo links.

\begin{theorem}\label{th:an-bracket}
Let $L$ be an oriented annular pseudo link diagram. The normalized annular pseudo bracket polynomial  is a 3-variable Laurent polynomial in the ring $\mathbb{Z}\left[A^{\pm 1},V,s\right]$, defined as:
$$P_L(A,V,s)\ =\ (-A^{-3})^{w(L)}\, \langle L \rangle_\mathcal{A}$$
\noindent where $\langle L \rangle_\mathcal{A}$ denotes the annular pseudo bracket polynomial of $L$ for $H = 1-Vd$, and it is an isotopy invariant of annular pseudo links and their corresponding pseudo links in the solid torus. Likewise for the universal normalized annular pseudo bracket polynomial $P_{{\rm O}\cup L}(A,V,s_k)$ for $k\in \mathbb{N}$. 
\smallbreak
\noindent Furthermore, for ${\rm O}\cup L$ an oriented ${\rm O}$-mixed pseudo link diagram, the normalized ${\rm O}$-mixed pseudo bracket polynomial  is a 3-variable Laurent polynomial in the ring $\mathbb{Z}\left[A^{\pm 1},V,s\right]$, defined as:

\[
P_{{\rm O}\cup L}(A,V,s)\ =\ (-A^{-3})^{w({\rm O}\cup L)}\, \langle {\rm O}\cup L \rangle,
\]

\noindent where $\langle{\rm O}\cup L \rangle$ denotes the ${\rm O}$-mixed pseudo bracket polynomial of ${\rm O}\cup L$ for $H = 1-Vd$, and it is an isotopy invariant of ${\rm O}$-mixed pseudo links  and their corresponding spatial ${\rm O}$-mixed pseudo links. Likewise for the universal normalized  ${\rm O}$-mixed pseudo bracket polynomial $P_{{\rm O}\cup L}(A,V,s_k)$ for $k\in \mathbb{N}$. 
\end{theorem}

\begin{proof}
As in the case of the planar pseudo bracket polynomial (Theorem~\ref{th:pl-bracket}), multiplying the annular pseudo bracket polynomial  with the  writhe correction factor $\left( -A^{-3}\right) ^{w(L)}$ ensures invariance under Reidemeister moves R1, while retaining invariance under regular isotopy, since the writhe is invariant of regular isotopy. For the invariance under the move PR1, the specialization $H = 1-Vd$ must be forced, as argued in \cite{HD} and discussed in Subsection~\ref{sec:planar-Jones}. So, the polynomial $P$ is an invariant of annular pseudo links, analogous to the normalized pseudo bracket for planar pseudo links. 

The invariance of the normalized ${\rm O}$-mixed pseudo bracket polynomial follows analogously from the definition of ${\rm O}$-mixed pseudo bracket (Theorem~\ref{th:an-bracket}), the definition of the ${\rm O}$-writhe (Definition~\ref{def:annular_writhe} and Theorem~\ref{th:pl-bracket} for planar pseudo links. 
\end{proof}

\begin{definition}
Applying the variable substitution $A= t^{-1/4}$ in the polynomial $P_L(A, V, s)$ we obtain the equivalent {\it annular pseudo Jones polynomial}, $P_L(t^{-1/4}, V, s)$, which is an isotopy invariant of annular pseudo links and their corresponding pseudo links in the solid torus. Similarly, the same variable substitution in the  normalized ${\rm O}$-mixed pseudo bracket polynomial  $P_{{\rm O}\cup L}(A,V,s)$ gives rise to the \textit{${\rm O}$-mixed pseudo Jones polynomial}, $P_{{\rm O}\cup L}(t^{-1/4},V,s)$, an isotopy invariant of ${\rm O}$-mixed pseudo links and their spatial analogues. 
\end{definition}

\begin{remarks}
\,\\
(a) As in the case of planar pseudo links, the  variable $V$ witnesses the presence of precrossings. In the absence of precrossings, the polynomials $P_L(A, V, s)$ and $P_{{\rm O}\cup L}(A,V,s)$ specialize to the analogue of the normalized  bracket (or the Jones polynomial for $A= t^{-1/4}$) for links in the solid torus (cf. \cites{HP,La2}). 
\smallbreak
\noindent (b) The new variable $s$ in the annular pseudo bracket resp. the ${\rm O}$-mixed pseudo bracket, which remains unchanged in the normalized  polynomials $P_L(A, V, s)$ and $P_{{\rm O}\cup L}(A,V,s)$, distinguishes essential annular pseudo links from non-essential ones (i.e. ones that can be isotoped in an inclusion disc). Clearly, for non-essential annular pseudo links, the annular normalized  polynomials $P_L(A, V, s)$ and $P_{{\rm O}\cup L}(A,V,s)$ specialize to the corresponding planar one  $P_L(A, V)$.
\end{remarks}

In view of the above remarks, the reader may compare the polynomials of the planar pseudo trefoil of Example~\ref{exampleptrefoil} and its essential analogue of Example~\ref{exampleannularptrefoil}, as well as of the classical trefoil or the unknot and their essential analogues.

\section{The theory of toroidal pseudo links}\label{sec:toroidal}

In this section, we recall essential results on toroidal pseudo links as presented in  \cite{DLM3}. In particular we recall the bijections between toroidal pseudo links and their lift in the thickened torus, as well as with the  ${\rm H}$-mixed pseudo links in the three-sphere. We then define the bracket polynomial for toroidal pseudo links and for ${\rm H}$-mixed  pseudo links and prove that they are invariants of regular isotopy. We also normalize these polynomials  to the corresponding Jones-type polynomials, which we prove to be topological invariants of toroidal pseudo links, pseudo links in the thickened torus and ${\rm H}$-mixed  pseudo links in the three-sphere, respectively.

\subsection{Preliminaries}\label{sec:prel-toroidal}

A {\it toroidal pseudo link diagram} is defined as a pseudo link diagram in the torus $T^2$. An example is illustrated in Figure~\ref{pkttorus}. A \textit{toroidal pseudo link} is an equivalence class of toroidal pseudo link diagrams related by a finite sequence of surface isotopy moves in $T^2$, the classical Reidemeister moves and the pseudo Reidemeister moves (recall Figure~\ref{reid}). This equivalence relation between toroidal pseudo link diagrams is called {\it isotopy}. Further, according to Definition~\ref{def:pseudo-regular}, {\it regular isotopy } among toroidal pseudo link diagrams is induced by all the above moves with the exception of the moves R1 and PR1.

\begin{figure}[H]
\begin{center}
\includegraphics[width=1.8in]{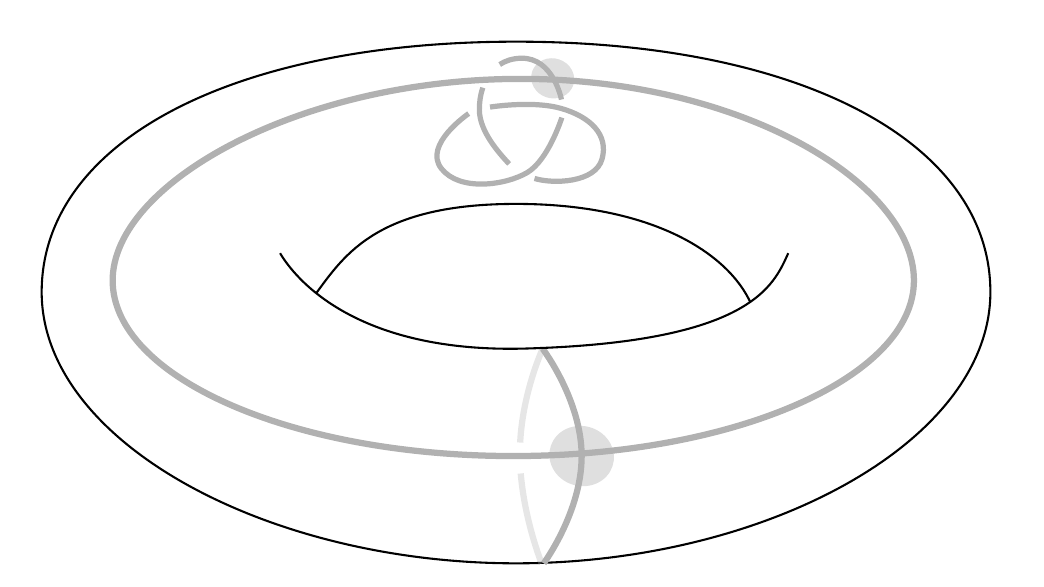}
\end{center}
\caption{A toroidal pseudo link.}
\label{pkttorus}
\end{figure}

In analogy to the theory of annular pseudo links, toroidal pseudo links can be studied in different topological contexts. Recall first that the {\it lift} of a toroidal pseudo link diagram in the thickened torus $T^2 \times I$, where $I$ denotes the unit interval $[0,1]$, is defined as a collection of closed curve(s) constrained to the interior of $T^2 \times I$ consisting of embedded discs from which emanate embedded arcs \cite{DLM3}. More specifically, each classical crossing of the diagram is embedded in a sufficiently small 3-ball that lies entirely within $T^2 \times I$, while precrossings are supported by sufficiently small rigid discs, which are embedded in $T^2 \times I$. Note also that the simple arcs connecting the (pre)crossings can be replaced by isotopic ones in $T^2 \times I$. The resulting lift is called a {\it pseudo link in the thickened torus}. For an illustration of the lift of a toroidal pseudo link diagram to a pseudo link in the thickened torus, we refer to Figure~\ref{pthtor}. 

\begin{figure}[H]
\begin{center}
\includegraphics[width=2.2in]{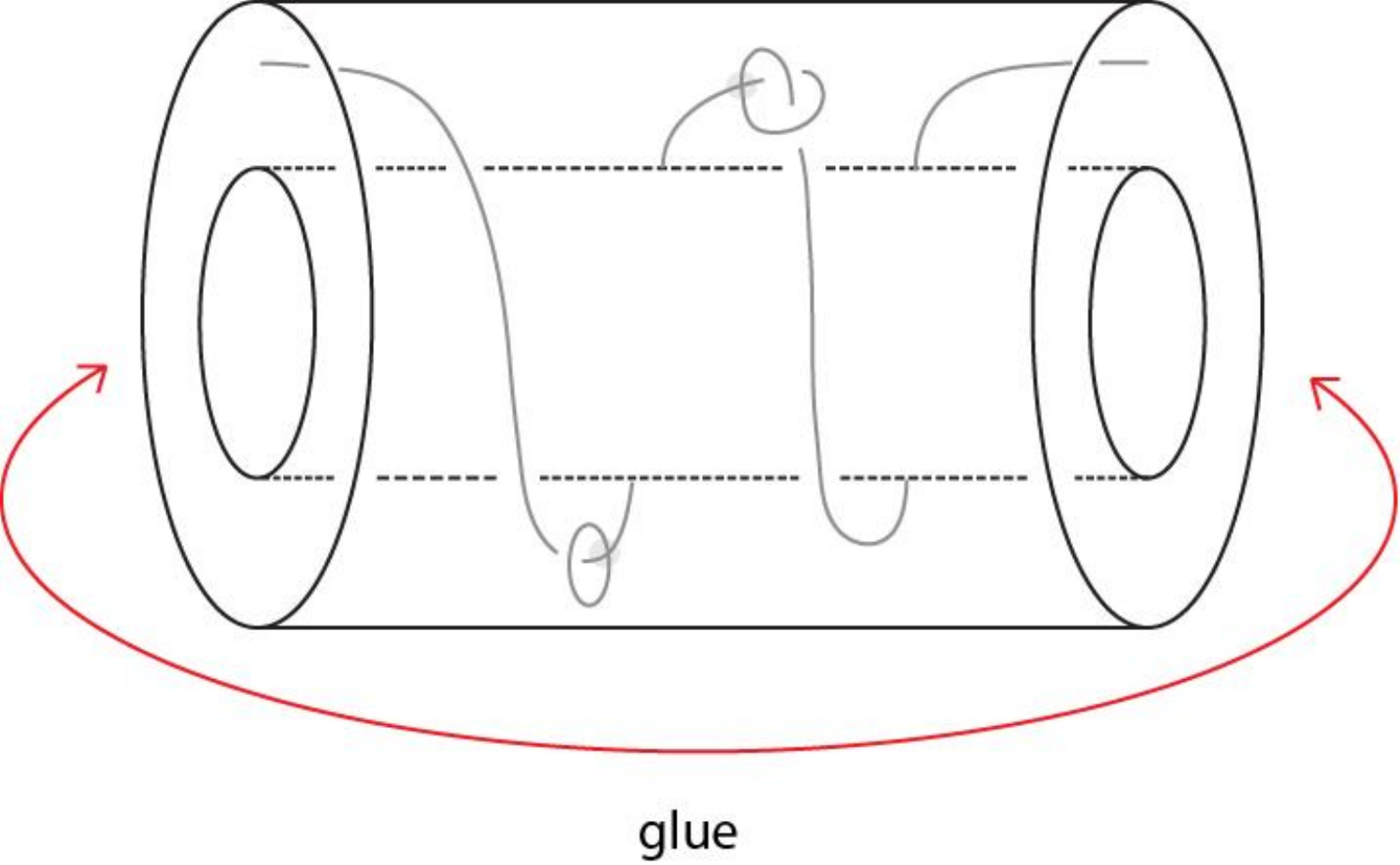}
\end{center}
\caption{The lift of a toroidal pseudo link to a pseudo link in $T^2 \times I$.}
\label{pthtor}
\end{figure}

Two (oriented) pseudo links in the thickened torus are said to be {\it (ambient) isotopic} if they are related by isotopies of arcs and discs that are confined to the interior of the thickened torus. In \cite{DLM3} we proved the following theorem:

\begin{theorem}\label{isopTT}
Two (oriented) pseudo links in the thickened torus $T^2 \times I$ are (ambient) isotopic if and only if any two corresponding toroidal pseudo link diagrams of theirs, projected onto the torus $T^2 \times \{0\}$, are related by surface isotopies, classical Reidemeister moves and the pseudo Reidemeister moves.
\end{theorem}

Another approach to the theory of toroidal pseudo links is to represent toroidal pseudo link diagrams, resp. pseudo links in the thickened torus, as mixed pseudo link diagrams in the plane, resp. as mixed pseudo links in the three-sphere $S^3$, using a similar approach as the one used to study annular pseudo links. In the case of pseudo links in the thickened torus,  we can see the (fixed) thickened torus as being homeomorphic to the complement of the Hopf link, ${\rm H}$ in the three-sphere $S^3$, with marked components (say $m$ and $l$) as illustrated in the middle of Figure~\ref{tor}. Then, an (oriented) pseudo link in $T^2 \times I$ can be represented uniquely by an (oriented) mixed link in $S^3$, whose fixed part is the (marked) Hopf link ${\rm H}$, representing the thickened torus. For an example of a toroidal pseudo link and its corresponding ${\rm H}$-mixed pseudo link, we refer to Figure~\ref{tor}. 

\begin{figure}[H]
\begin{center}
\includegraphics[width=4.5in]{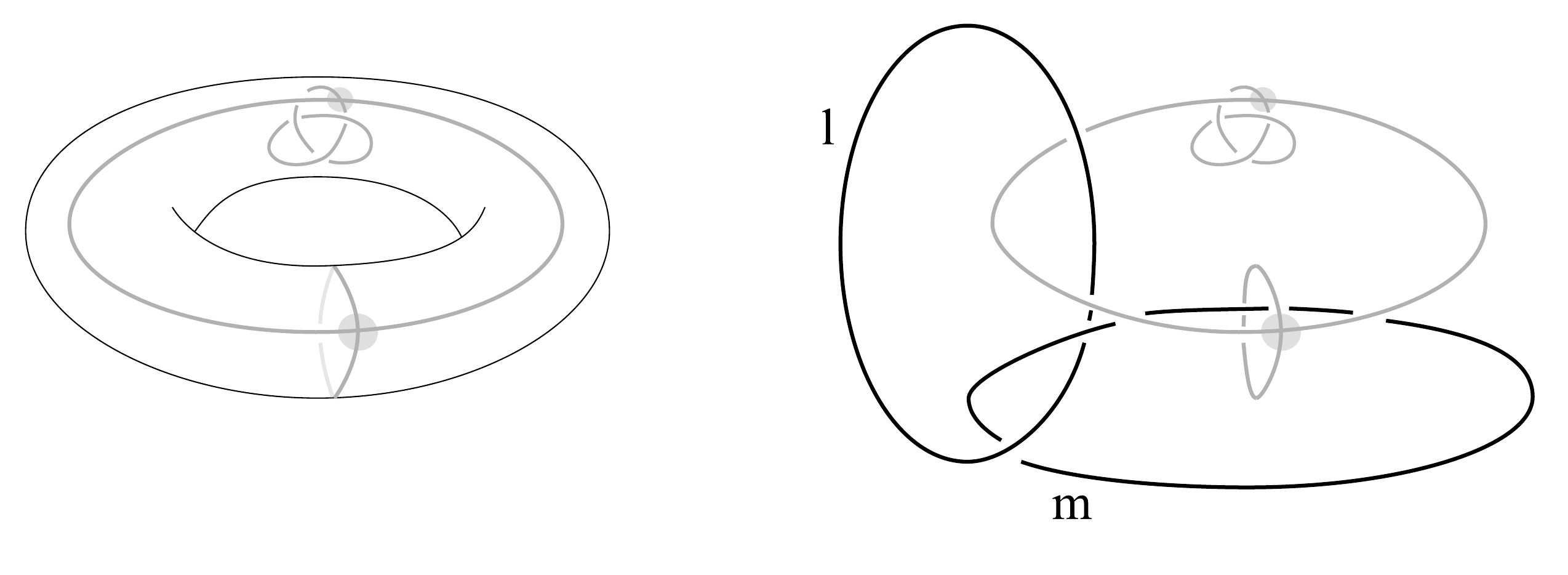}
\end{center}
\caption{A toroidal pseudo link and its corresponding ${\rm H}$-mixed pseudo link.}
\label{tor}
\end{figure}

The isotopy of ${\rm H}$-mixed pseudo links can be translated on the diagrammatic level, recalling that the fixed part of the ${\rm H}$-mixed pseudo links involves a crossing, which a moving strand can freely cross, giving rise to an extra mixed Reidemeister 3 move with one moving arc, as exemplified in Figure~\ref{mr3}. Thus, it follows that:

\begin{theorem}[\cite{DLM3}] \label{Hmixedreid}
Two oriented ${\rm H}$-mixed pseudo links in $S^{3}$ are isotopic if and only if any two oriented ${\rm H}$-mixed pseudo link diagrams of theirs differ by planar isotopies, a finite sequence of the classical and the pseudo Reidemeister moves for the moving parts of the mixed pseudo links (recall Figure~\ref{reid}), and moves that involve the fixed and the moving parts, called {\rm mixed Reidemeister moves}, comprising the moves MR2, MR3, and MPR3, as exemplified in Figure~\ref{mpr} and the mixed R3 moves, as exemplified in Figure~\ref{mr3}. 
\end{theorem}

\begin{figure}[H]
\begin{center}
\includegraphics[width=2.5in]{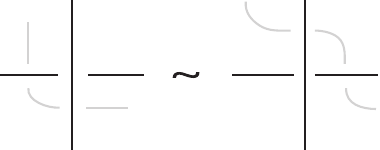}
\end{center}
\vspace{8pt}
\caption{A mixed R3 move.}
\label{mr3}
\end{figure}

It follows from the above that, {\it two oriented toroidal pseudo links diagrams are isotopic if and only if any two corresponding oriented ${\rm H}$-mixed pseudo link diagrams  are isotopic.}

\subsection{The toroidal pseudo bracket polynomial}

The pseudo bracket polynomial, defined as a 3-variable Laurent polynomial for planar pseudo links in Definition~\ref{pkaufb} and as a 4-variable Laurent polynomial for annular pseudo links in Definition~\ref{pkaufbst}, can also be extended to the setting of toroidal pseudo links. These are pseudo links whose diagrams lie in the surface of the torus $T^2$ and lift in the thickened torus. The main difference from the above two polynomials lies in the fact that when applying the skein relations to toroidal pseudo link diagrams, the essential unknotted curves that can appear in a state diagram are torus links. 

\smallbreak
A \textit{$(p,q)$-torus knot}, for $p$ and $q$ relatively prime integers, is an essential simple closed curve in the torus $T^2$ winding $p$ times around the longitude and $q$ times along the meridian of the torus. Similarly, a \textit{$(p',q')$-torus link} or a \textit{$(p,q)$-torus link on $l$ components} is a set of $l$ parallel copies of a $(p,q)$-torus knot, where $p'= l \cdot p$ and $q'= l \cdot q$. In our definition of torus knots and links, we include the pairs $(p,0)$ and $(0,q)$. 

\begin{remark}\label{rem:essentialpq}
A $(p,q)$-torus knot, considered as a classical knot in $S^3$, is equivalent a $(q,p)$-torus knot.  However, as essential curves embedded in the torus, these two knots are distinct in $T^2$. Yet, note that in $T^2$ a $(p,q)$-torus knot is the same as a $(-p,-q)$-torus knot, and similarly, a $(p,-q)$-torus knot is the same as a $(-p,q)$-torus knot. 
\end{remark}

All torus knots and links along with the standard null-homotopic unknot constitute initial conditions for the pseudo bracket polynomial of toroidal pseudo link diagrams. So, we assign a specific variable $s_{p,q}$ to each $(p,q)$-torus knot and we define the following: 

\begin{definition}\label{kbkntt}
Let $K$ be an oriented toroidal pseudo link diagram. The \textit{toroidal pseudo bracket polynomial} of $K$, denoted $\langle K \rangle_{T^2}$ or simply $\langle K \rangle$, is an infinite variable Laurent polynomial in the ring $\mathbb{Z}\left[A^{\pm 1}, V, H, s_{p,q} \right]$, for $p$ and $q$ pairs of coprime integers with $q$ non negative, defined inductively by means of the following rules, or $(p, q) \in \{(1,0), (0, 1)\}$, where a punctured plane in the 5th rule represents the torus:

\begin{figure}[H]
\begin{center}
\includegraphics[width=4.5in]{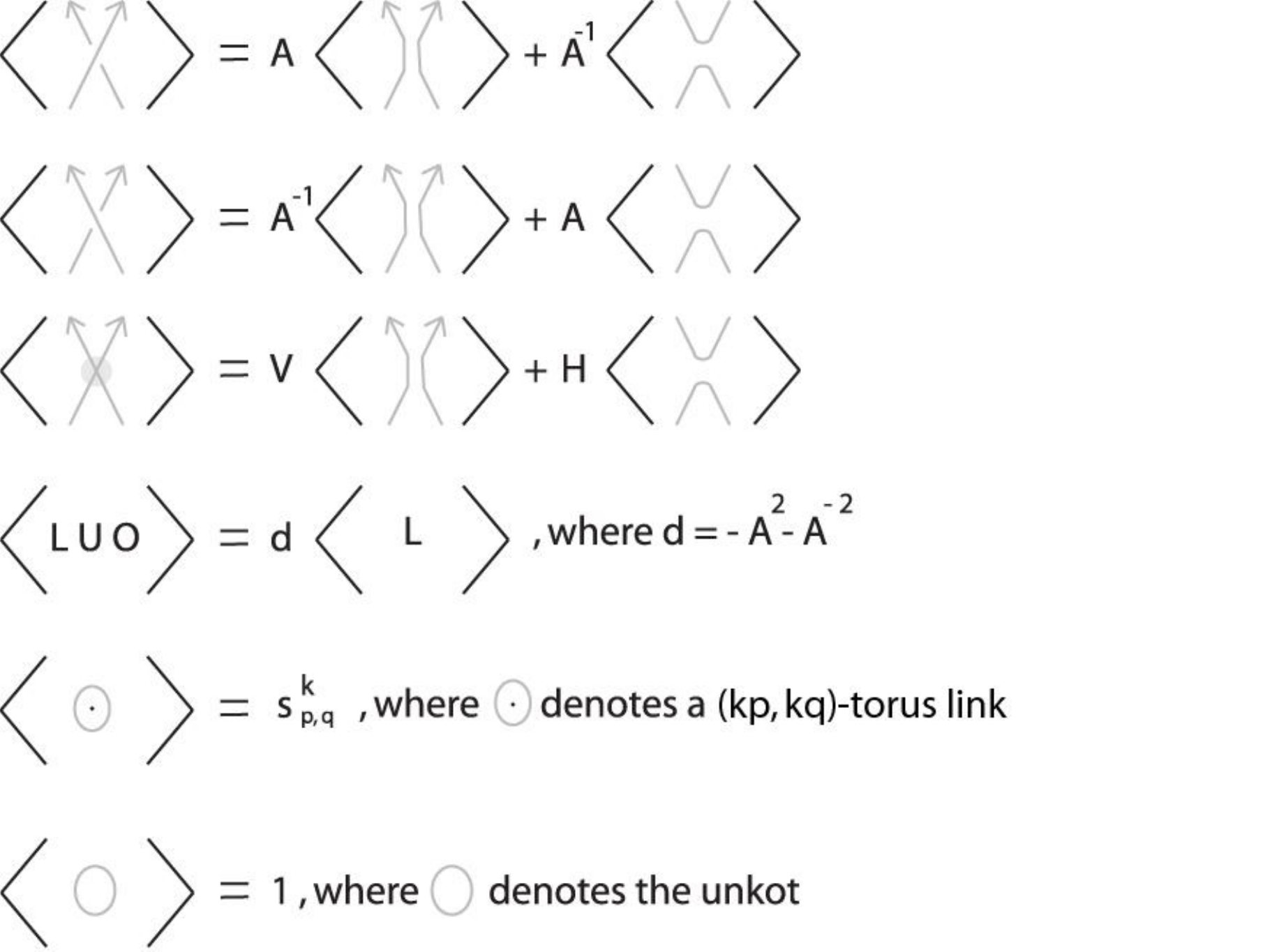}
\end{center}
\label{pkbstor}
\end{figure}
\end{definition}

\noindent  Before proceeding, we define a simplified version of the toroidal pseudo bracket polynomial that can  be obtained by introducing two new variables $x$ and $y$ in place of $s_{p,q}$ in the 5th rule, which can be useful in computations. More precisely: 

\begin{definition}\label{def:reduced-tor-bracket}
Let $K$ be an oriented toroidal pseudo link. The \textit{reduced toroidal pseudo bracket polynomial} of $K$, denoted $\langle K \rangle^r_{T^2}$ or simply $\langle K \rangle^r$, is a 5-variable Laurent polynomial in the ring $\mathbb{Z}\left[A^{\pm 1}, V, H, x^{\pm 1}, y \right]$, defined inductively by means of the same rules as for the polynomial $\langle K \rangle_{T^2}(A, V, H, s_{p,q})$, except for the fifth rule, in which $s_{p,q}$ is replaced by $x^p y^q$, for $p$ and $q$ pairs of coprime integers with $q$ non negative, or $(p, q) \in \{(1,0), (0, 1)\}$. 
We further define the {\it universal toroidal pseudo bracket polynomial} by substituting in the fifth rule $s_{p,q}^k$ by $s_{p,q,k}$.
\end{definition}

\begin{remark}
Since the precrossings are smoothened, the universal toroidal pseudo bracket polynomial recovers the Kauffman bracket skein module of toroidal pseudo links. Cf. \cite{BaPr} and references therein. 
\end{remark} 

Note that the last rule can only be applied to a toroidal state diagram that contains only a standard null-homotopic unknot. We also point out that two torus knot components of a toroidal pseudo link diagram, which have different slopes, will necessarily form crossings, which will be smoothened by the skein relations of Definition~\ref{kbkntt} (and Definition~\ref{def:reduced-tor-bracket} for the reduced case). So, a state diagram of the toroidal pseudo bracket polynomial can only contain a single torus knot or link. Therefore, together with the comment above Definition~\ref{kbkntt} we have that the toroidal pseudo bracket polynomial and the reduced one are well defined.

\smallbreak
In analogy to the planar and annular pseudo bracket polynomials of Definition~\ref{pkaufb} and Definition~\ref{pkaufbst}, we have the following:

\begin{theorem}\label{th:regular_invariance_tor}
The (universal) toroidal pseudo bracket polynomial and the reduced toroidal pseudo bracket polynomial are invariant under regular isotopy of oriented toroidal pseudo links.
\end{theorem} 

\begin{proof}
The invariance under the Reidemeister moves R2, R3, and the pseudo Reidemeister moves PR2, PR3 follows by similar arguments as in the proof of Proposition~\ref{annular_regular_invariance} for the invariance under regular isotopy in the annular case, and by the locality of these moves. 
\end{proof}

\begin{remarks}\label{rem:inclusions-tor}
In  \cite{DLM3} we discuss inclusion relations of the pseudo knot theories of the various supporting manifolds involved. Based on these inclusion relations we now discuss  relations of the corresponding pseudo bracket polynomials.

\noindent (a) If a toroidal pseudo diagram contains no essential closed curves, by using the inclusion of a disc in $T^2$ (see left-hand illustration in Figure~\ref{planarannulartorus}), such toroidal diagrams can be viewed as planar  diagrams and resolve into states as in the planar case.

\begin{figure}[H] 
\begin{center} 
\includegraphics[width=6in]{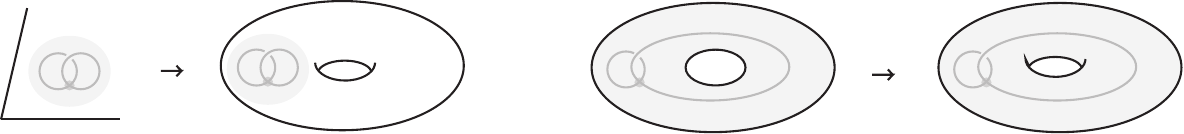} 
\end{center} 
\vspace{8pt}
\caption{Inclusion relations: a disc in the torus and an annulus in the torus.} 
\label{planarannulartorus} 
\end{figure} 

\noindent (b) By the same reasoning, if a toroidal pseudo diagram can be isotoped to one that contains no meridional windings but only longitudinal ones, by using the inclusion of an annulus in $T^2$ (see right-hand illustration in Figure~\ref{planarannulartorus}), such toroidal diagrams can be viewed as annular diagrams and resolves into states as in the annular case.

\smallbreak
\noindent  (c) In contrast, using the inclusion of the thickened torus in a 3-ball (see Figure~\ref{inclusion-tor-disk.pdf}),  the toroidal pseudo bracket polynomial specializes to the planar pseudo bracket of Definition~\ref{pkaufb} for planar pseudo links. Therefore, a $(p,q)$-torus knot will be treated as a classical knot, so it will resolve further according to the rules of the classical bracket polynomial. In particular, an $(1,0)$-torus knot or a $(0,1)$-torus knot become null-homotopic, so  $s_{1,0}^k$ and $s_{0,1}^k$ will be substituted by $d^{k-1} = (-A^2-A^{-2})^{k-1}$. View also Figure~\ref{allinclusions}.

\begin{figure}[H] 
\begin{center} 
\includegraphics[width=3.8in]{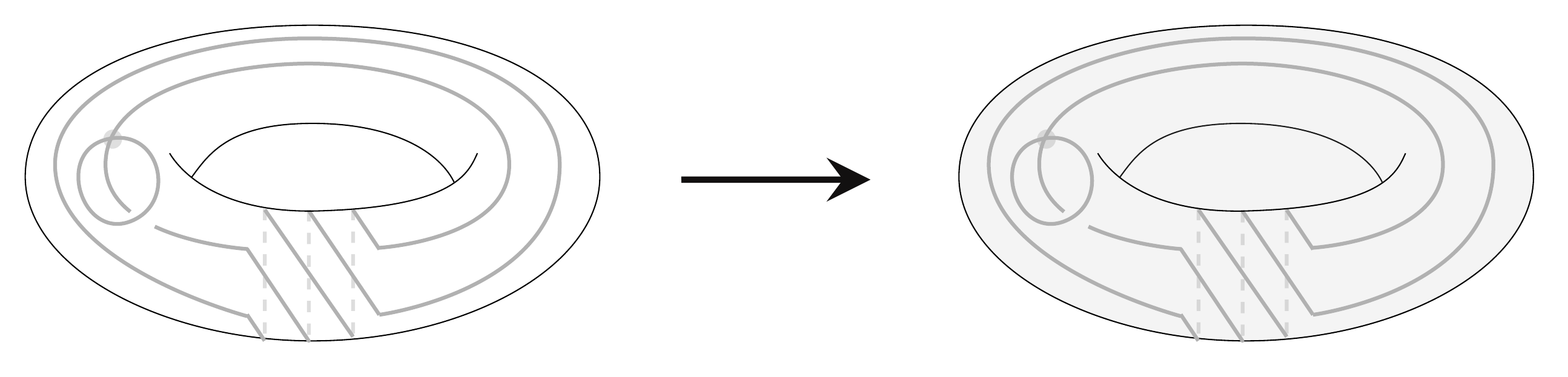} 
\end{center} 
\caption{Inclusion relation of the thickened torus in a 3-ball.} 
\label{inclusion-tor-disk.pdf} 
\end{figure} 

\noindent (d) By the same reasoning, using the inclusion of the thickened torus in the solid torus (see Figure~\ref{inclusion-ttor-st}),  the toroidal pseudo bracket polynomial specializes to the annular pseudo bracket of Definition~\ref{pkaufbst} for annular pseudo links. Therefore, a $(p,q)$-torus knot will be treated as an annular knot, so it will resolve further according to the rules of the annular  bracket polynomial. In particular, an $(1,0)$-torus knot becomes a null-homotopic unknot in the annulus, whilst a $(0,1)$-torus knot becomes an essential unknot in the annulus, so $s_{0,1}^k$ will be substituted by $s^k$. Similarly, in the reduced toroidal pseudo bracket polynomial we will have $(x^0 y^1)^k =s^k$. View also Figure~\ref{allinclusions}. 
\end{remarks}

\begin{figure}[H]
\begin{center}
\includegraphics[width=3.5in]{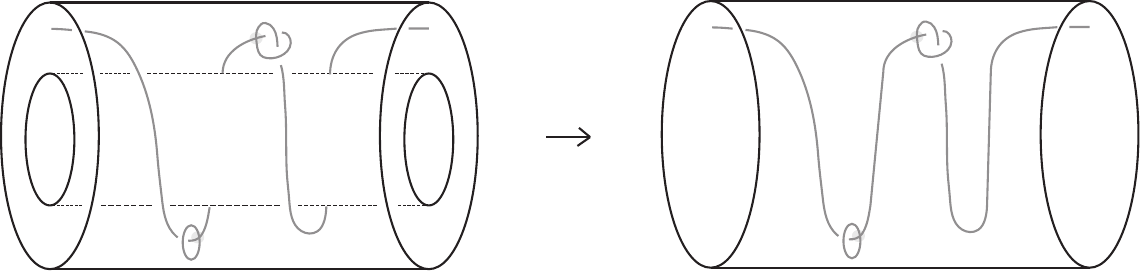}
\end{center}
\vspace{5pt}
\caption{Inclusion relation of the thickened torus in the solid torus.}
\label{inclusion-ttor-st}
\end{figure}

\begin{figure}[H] 
\begin{center} 
\includegraphics[width=5in]{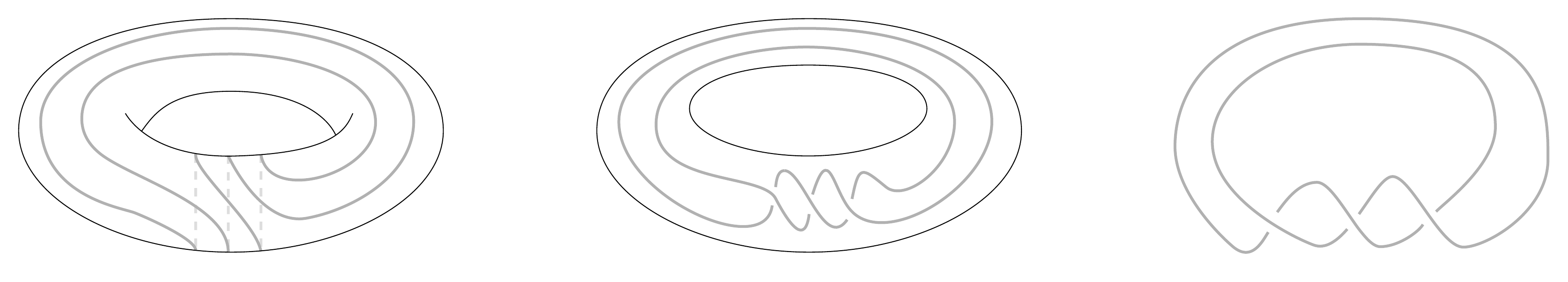} 
\end{center} 
\caption{A $(3,2)$-torus knot becomes an essential trefoil knot in the annulus and  a trefoil knot in the plane.} 
\label{allinclusions} 
\end{figure} 

\begin{remark}
Consider a torus link consisting of two different torus knot components, say a $(1,0)$-torus knot and a $(0,1)$-torus knot which intersect in a precrossing. Assigning orientations to the two components we observe that reversing the orientation of one component will result in different bracket polynomials, due to the fact that the roles of $V$ and $H$ in the third rule of the toroidal pseudo bracket are not interchangeable.  The situation is illustrated in Figure~\ref{reverseorientation}. The left-hand side toroidal pseudo link, $L_1$, is assigned the polynomial:
$$\langle L_1 \rangle_{T^2} = V \, s_{1,1} + H \, s_{-1,1},$$ 
\noindent whilst the right-hand side toroidal pseudo link, $L_2$, is assigned the polynomial: 
$$\langle L_2 \rangle_{T^2} = V \, s_{-1,1} + H \, s_{1,1}.$$ 
\noindent In the respective reduced versions: 
$$\langle L_1 \rangle^r_{T^2} = V x y + H x^{-1} y$$ 
\noindent and: 
$$\langle L_2 \rangle^r_{T^2} = V x^{-1} y + H x y.$$
\end{remark} 

\begin{figure}[H] 
\begin{center} 
\includegraphics[width=5.5in]{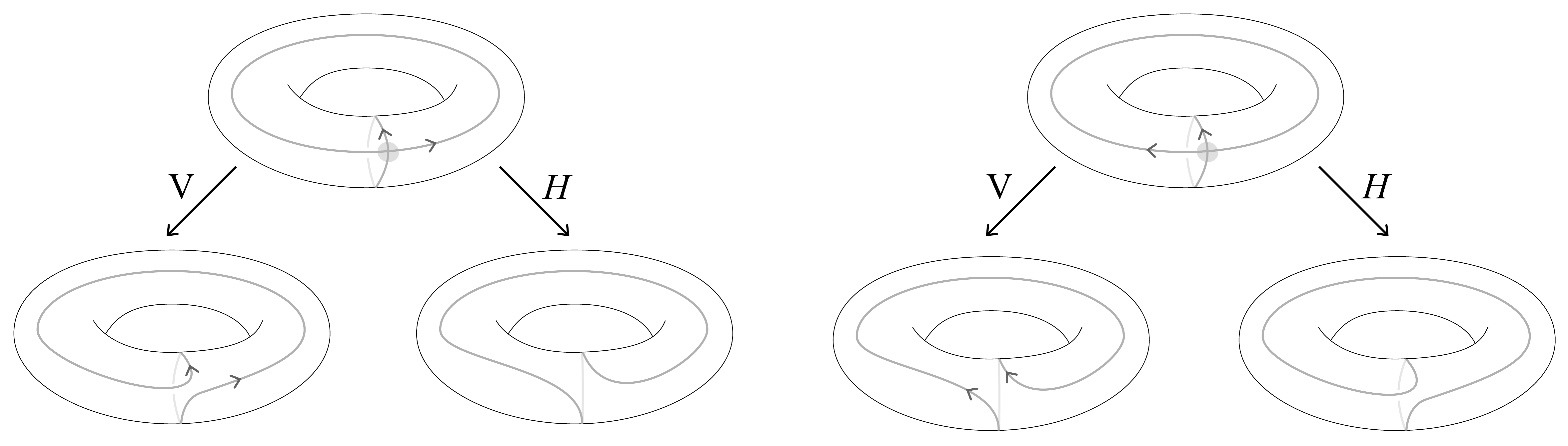} 
\end{center} 
\vspace{8pt}
\caption{The orientation sensitive skein analysis for the toroidal pseudo links $L_1$ and $L_2$.} 
\label{reverseorientation} 
\end{figure} 

\begin{example}\label{ex:toroidal}
Consider the two different embeddings of a pseudo trefoil knot in the thickened torus, whose diagrams are depicted in the top row of Figure~\ref{tortref}. The one on the top left side, denoted $K_l$, has the same skein tree is the same as in Figure~\ref{pseudo-annular} for the annular case due to the inclusion relation of the annulus in the torus (recall Remark~\ref{rem:inclusions-tor} (d)). Hence, we have: 
$$ \langle K_l \rangle_{T^2} = - V^2 s_{0,1}^2 A^3 - VHA^{-3} + 2VHA^{-1} + V^2 A^{-1} - V^2 s_{0,1}^2 A^{-1} - H^2 A^{-3} + VHA^{-5} $$

\noindent and the reduced version:
$$\langle K_l \rangle^r_{T^2} = - V^2 y^2 A^3 - VHA^{-3} + 2VHA^{-1} + V^2 A^{-1} - V^2 y^2 A^{-1} - H^2 A^{-3} + VHA^{-5}. $$

\begin{figure}[H]
\begin{center}
\includegraphics[width=4in]{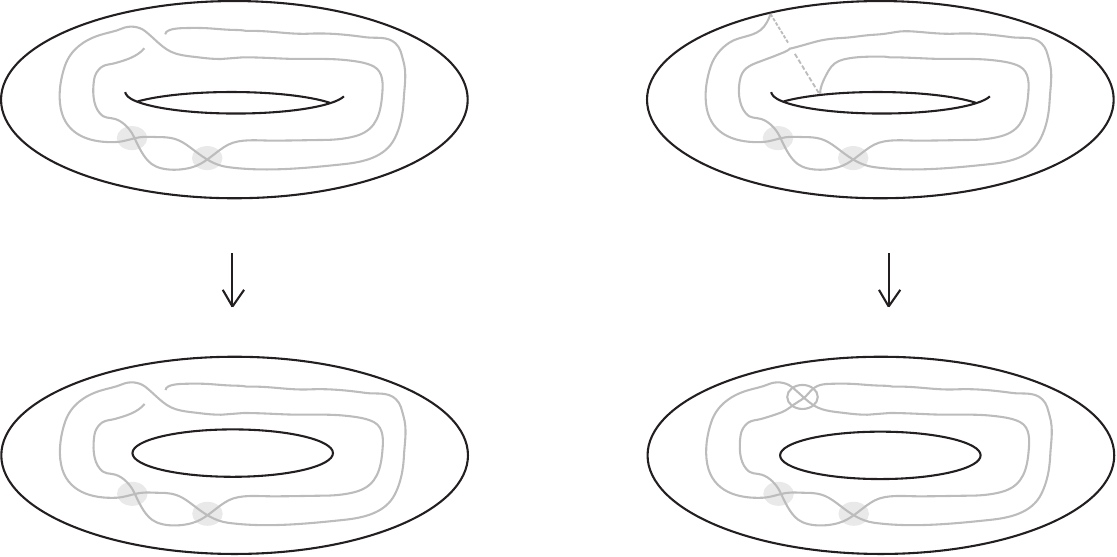}
\end{center}
\vspace{8pt}
\caption{Two toroidal pseudo trefoils projected on the annulus and the appearance of a virtual crossing (cf. \cite{DLM3}).}
\label{tortref}
\end{figure}

\noindent The pseudo trefoil on the top right side of the figure, denoted $K_r$, is not of annular type, so when projecting on the annulus, a {\it virtual crossing} will appear, as illustrated in Figure~\ref{tortref} bottom right, which is neither a real crossing  nor a precrossing (cf. \cite{LK2}). The use of virtual crossings in this context comes in handy (only) for illustration purposes, as discussed in \cite{DLM3}, in order to represent toroidal diagrams in the annulus or in the once punctured disc. We stress that the isotopy moves of virtual knot theory do not apply here. The skein tree analysis of $K_r$ is depicted in Figure~\ref{Virtual-Pseudo-toroidal}. 

\begin{figure}[H]
\begin{center}
\includegraphics[width=6in]{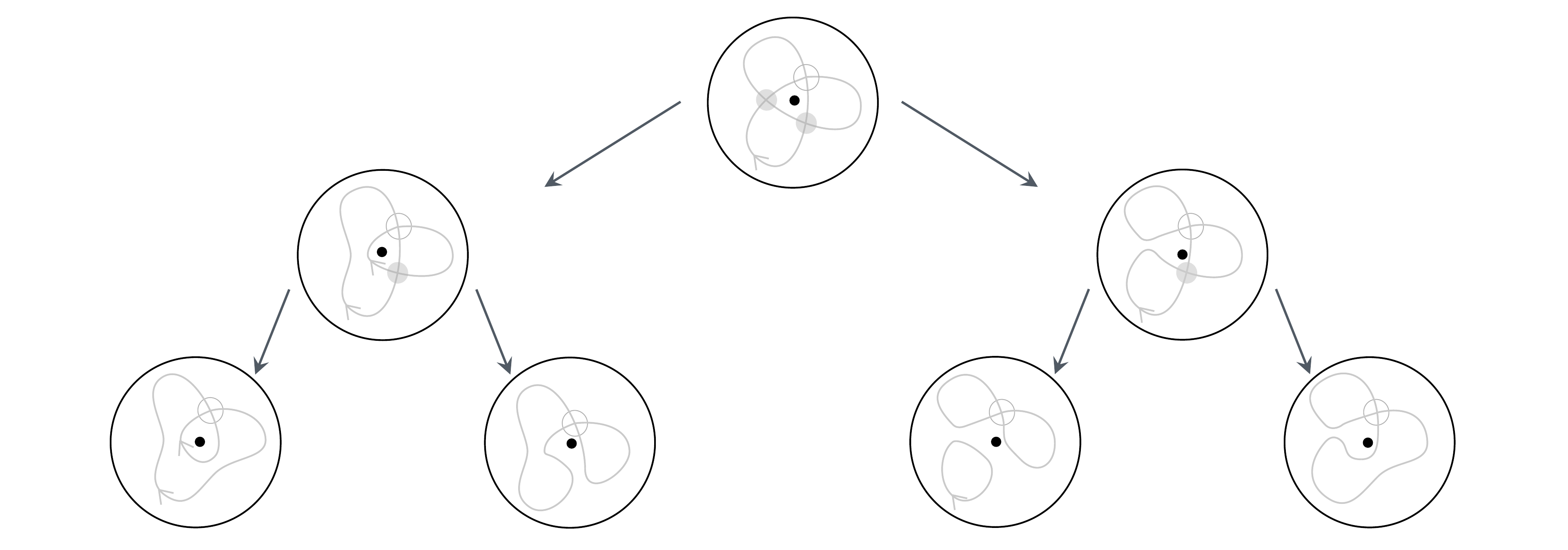}
\end{center}
\vspace{8pt}
\caption{The skein tree of the virtual toroidal pseudo trefoil depicted in Figure~\ref{tortref} top right.}
\label{Virtual-Pseudo-toroidal}
\end{figure}

\noindent So, we obtain the following:
$$\langle K_r \rangle_{T^2} = V^2 s_{1,2} + \big{(}VH (1 - A^2 - A^{-2}) + H^2\big{)} s_{1,0} $$

\noindent and the reduced version:
$$\langle K_r \rangle^r_{T^2} = V^2 x y^2 + \big{(}VH (1 - A^2 - A^{-2}) + H^2\big{)} x $$

\noindent where the rightmost state diagram consists of a $(1,2)$-torus knot and the three other state diagrams contain one meridional curve (in the third diagram we take out the isolated unknot by a factor $d$). 
\end{example}

\subsection{The ${\rm H}$-mixed pseudo bracket polynomial}

In this subsection, we consider the planar ${\rm H}$-mixed pseudo link approach to toroidal pseudo links (recall Subsection~\ref{sec:prel-toroidal}). In view of Theorem~\ref{Hmixedreid}, the toroidal pseudo bracket polynomial can be adapted to the setting of oriented ${\rm H}$-mixed pseudo links. Here the crossings of ${\rm H}$  as well as the mixed crossings are not subjected to skein relations, since ${\rm H}$ represents the thickened torus, so must remain fixed throughout. 

We now consider the torus links of type $(kp,kq)$, which form the states for the toroidal pseudo bracket polynomial of Definition~\ref{kbkntt}, and which we want to correspond to ${\rm H}$-mixed  links. Figure~\ref{Hmixedtorusknot} illustrates a $(3,2)$-torus knot represented in two ways as an ${\rm H}$-mixed link. The left-hand illustration of the figure uses the common representation of ${\rm H}$, while in the right-hand illustration ${\rm H}$ is represented in closed braid form. As we observe, in the left-hand illustration appear three moving crossings, which should be smoothened according to the bracket rules, a situation we can avoid using the representation of the $(3,2)$-torus knot in the right-hand illustration. 

\begin{figure}[H]
    \centering
    \includegraphics[width=4in]{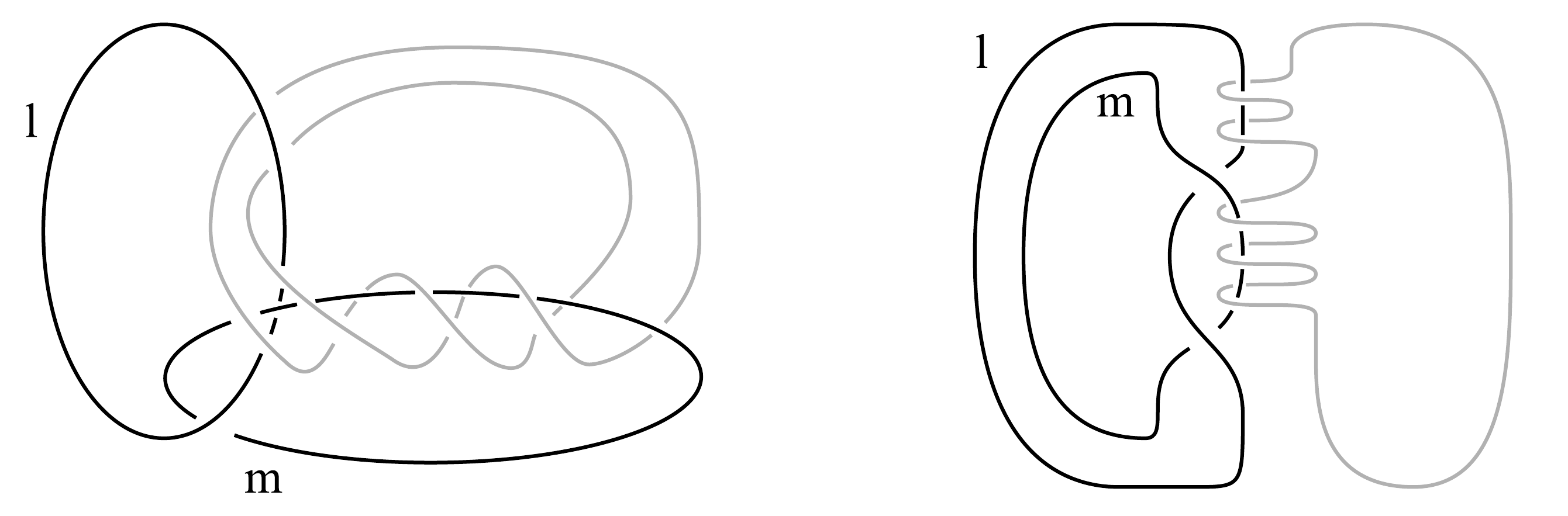}
    \vspace{8pt}
    \caption{ A $(3,2)$-torus knot represented in two ways as an ${\rm H}$-mixed link.}
    \label{Hmixedtorusknot}
\end{figure}

We are now ready to translate the toroidal pseudo bracket polynomial in the ${\rm H}$-mixed setting:

\begin{definition}\label{mix-pkaufbtt}\rm
Let $K$ be a toroidal oriented pseudo link diagram and let ${\rm H}\cup K$ the corresponding ${\rm H}$-mixed pseudo link diagram, with ${\rm H}$ being represented in closed braid form. The \textit{${\rm H}$-mixed pseudo bracket polynomial} of ${\rm H}\cup K$ is an infinite variable Laurent polynomial in the ring $\mathbb{Z}\left[A^{\pm 1}, V, H, s_{p,q} \right]$,  defined by means of the same inductive rules as the ones for the toroidal mixed pseudo bracket polynomial, except for the  $(kp,kq)$-torus link diagram in the fifth rule of Definition~\ref{kbkntt}, which is substituted by the corresponding ${\rm H}$-mixed link diagram, as illustrated in Figure~\ref{Hmixedtoruslink}. Also, $L$ in the fourth rule stands now for an oriented ${\rm H}$-mixed pseudo link diagram.  We further define the \textit{universal ${\rm H}$-mixed pseudo bracket polynomial} by substituting in the fifth rule $s_{p,q}^k$ by $s_{p,q,k}$.
\end{definition}

\begin{figure}[H]
    \centering
    \includegraphics[width=2in]{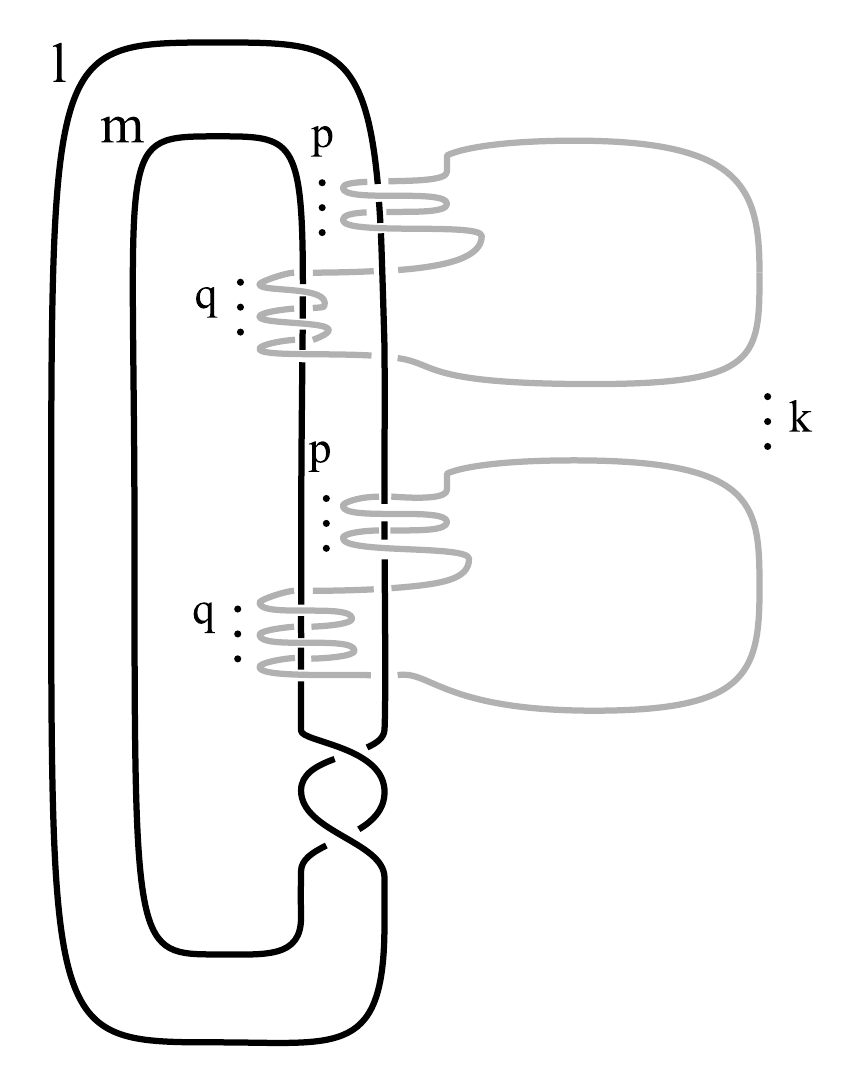}
    \caption{ An ${\rm H}$-mixed link diagram corresponding to the $(kp,kq)$-torus link.}
    \label{Hmixedtoruslink}
\end{figure}

\begin{remark} 
Diagrams of the form of Figure~\ref{Hmixedtoruslink} have been used in the braid approach to the skein modules of various 3-manifolds.  Note that these diagrams comprise only a subset of the complete theory of ${\rm H}$-mixed links, where different torus links may be present in a state diagram. For details cf. \cites{D2} and  references therein. 
\end{remark}

We may also define a simplified version of the ${\rm H}$-mixed pseudo bracket polynomial that can  be obtained by introducing two new variables $x$ and $y$ in place of $s_{p,q}$ in the 5th rule, which can prove handy in computations. More precisely: 

\begin{definition}\label{def:reduced-H-bracket}
Let ${\rm H}\cup K$ be an oriented ${\rm H}$-mixed pseudo link diagram, with ${\rm H}$ being represented in closed braid form. The \textit{reduced ${\rm H}$-mixed pseudo bracket polynomial} of ${\rm H}\cup K$ is a 5-variable Laurent polynomial in the ring $\mathbb{Z}\left[A^{\pm 1}, V, H, x^{\pm 1}, y \right]$, defined inductively by means of the same rules as for the ${\rm H}$-mixed pseudo bracket polynomial, except for the fifth rule, in which $s_{p,q}$ is replaced by $x^p y^q$, for $p$ and $q$ pairs of coprime integers with $q$ non negative, or $(p, q) \in \{(1,0), (0, 1)\}$.
\end{definition}

In analogy now to the annular case (Theorem~\ref{th:an-mix-in}), we then have for planar ${\rm H}$-mixed pseudo links the following:

\begin{theorem}\label{th:th-mix-in}
The (universal) ${\rm H}$-mixed and the reduced ${\rm H}$-mixed pseudo bracket polynomials are invariant under regular isotopy of oriented ${\rm H}$-mixed pseudo links and are equivalent to the  (universal) toroidal and the reduced toroidal pseudo bracket polynomials, respectively, for toroidal pseudo links. 
\end{theorem}

\begin{proof}
We consider the restriction of the theory of planar pseudo links to the subset of  ${\rm H}$-mixed pseudo links. Then, by Proposition~\ref{regular_invariance}, the (reduced) ${\rm H}$-mixed pseudo bracket polynomial is invariant under the classical Reidemeister moves R2 and R3 and under the pseudo Reidemeister moves PR2 and PR3. Regarding now the mixed Reidemeister moves MR2, MR3, MPR3 and mixed R3 (recall Figures~\ref{mpr} and~\ref{mr3}), since there is no rule for smoothing mixed crossings, the moves MR2 and the mixed R3 moves are undetectable by the (reduced) ${\rm H}$-mixed pseudo bracket, so it remains invariant. Further, invariance under the moves MR3 and MPR3 follows by the same arguments as in the annular case (recall Figure~\ref{kf1}). 

Finally, the equivalence of the (reduced) ${\rm H}$-mixed pseudo bracket polynomial for ${\rm H}$-mixed pseudo links to the (reduced) toroidal pseudo bracket polynomial for toroidal pseudo links follows immediately by Theorem~\ref{Hmixedreid}.
\end{proof}

\subsection{Jones-type analogues}

Similarly to the case of planar and annular pseudo links, we may normalize the toroidal and the ${\rm H}$-pseudo bracket polynomials in order to satisfy the Reidemeister moves R1 and PR1 and obtain  invariants for toroidal pseudo links and their corresponding pseudo links in the thickened torus resp. ${\rm H}$-mixed pseudo links. For this, we first define:

\begin{definition} \label{def:toroidal_writhe}
The \textit{writhe} of an oriented toroidal pseudo link diagram $L$, denoted as $w(L)$, is defined as the number of positive crossings minus the number of negative crossings of $L$ (recall Figure~\ref{sign}), while the precrossings do not contribute to the writhe. Furthermore, the ${\rm H}$-\textit{writhe} of an oriented ${\rm H}$-mixed pseudo link diagram ${\rm H}\cup L$, denoted $w({\rm H}\cup L)$, is defined as the number of positive crossings minus the number of negative crossings of $L$ (recall Figure~\ref{sign}), while the precrossings, the mixed crossings and the fixed crossings do not contribute to the ${\rm H}$-writhe. 
\end{definition}

We now have:

\begin{theorem}\label{th:tor-bracket}
Let $L$ be an oriented toroidal pseudo link diagram. The  (universal) normalized toroidal pseudo bracket polynomial is an infinite variable Laurent polynomial in the ring $\mathbb{Z}\left[A^{\pm 1},V,s_{p,q}\right]$ $\left( \text{resp.}\, \mathbb{Z}\left[A^{\pm 1},V,s_{p,q,k}\right]\right)$, defined as:
$$P_L(A,V,s_{p,q})\ =\ (-A^{-3})^{w(L)}\, \langle L \rangle_{T^2}$$
\noindent $\left( \text{resp.}\, P_L(A,V,s_{p,q,k})\right)$ where $\langle L \rangle_{T^2}$ denotes the (universal) toroidal pseudo bracket polynomial of $L$ for $H = 1-Vd$, and it is an isotopy invariant of toroidal pseudo links and their corresponding pseudo links in the thickened torus. 

\smallbreak
\noindent Similarly, the  normalized reduced toroidal pseudo bracket polynomial is 
 a 4-variable Laurent polynomial in the ring $\mathbb{Z}\left[A^{\pm 1}, V, x^{\pm 1}, y \right]$,  defined as: 
 
$$P^r_L(A,V,x,y)\ =\ (-A^{-3})^{w(L)}\, \langle L \rangle^r_{T^2}$$
\noindent where $\langle L \rangle^r_{T^2}$ denotes the reduced toroidal pseudo bracket polynomial of $L$ for $H = 1-Vd$, and it is an isotopy invariant of toroidal pseudo links and their corresponding pseudo links in the thickened torus. 

\smallbreak
\noindent Furthermore, for ${\rm H}\cup L$ an oriented ${\rm H}$-mixed pseudo link diagram corresponding to the toroidal pseudo link diagram $L$, the (universal) normalized ${\rm H}$-mixed pseudo bracket polynomial is an infinite variable Laurent polynomial in the ring $\mathbb{Z}\left[A^{\pm 1},V,s_{p,q}\right]$, defined as:
$$P_{{\rm H}\cup L}(A,V,s_{p,q})\ =\ (-A^{-3})^{w({\rm H}\cup L)}\, \langle {\rm H}\cup L \rangle,
$$
\noindent $\left( \text{resp. } P_{{\rm H}\cup L}(A,V,s_{p,q,k}) \text{ in the ring } \mathbb{Z}\left[A^{\pm 1},V,s_{p,q,k}\right]\right)$ where $\langle{\rm H}\cup L \rangle$ denotes the (universal) ${\rm H}$-mixed pseudo bracket polynomial of ${\rm H}\cup L$ for $H = 1-Vd$, and it is an isotopy invariant of ${\rm H}$-mixed pseudo links and their corresponding spatial ${\rm H}$-mixed pseudo links.

\smallbreak
\noindent Similarly, the  normalized reduced  ${\rm H}$-mixed pseudo bracket polynomial is a 4-variable Laurent polynomial in the ring $\mathbb{Z}\left[A^{\pm 1}, V, x^{\pm 1}, y \right]$,  defined as: 
$$P^r_{{\rm H}\cup L}(A,V,x,y)\ =\ (-A^{-3})^{w({\rm H}\cup L)}\, \langle {\rm H}\cup L \rangle^r,
$$
\noindent where $\langle{\rm H}\cup L \rangle^r$ denotes the reduced ${\rm H}$-mixed pseudo bracket polynomial of ${\rm H}\cup L$ for $H = 1-Vd$, and it is an isotopy invariant of ${\rm H}$-mixed pseudo links and their corresponding spatial ${\rm H}$-mixed pseudo links.
\end{theorem}

\begin{proof}
As in the case of the planar pseudo bracket  (Theorem~\ref{th:pl-bracket}) and the annular pseudo bracket polynomials (Theorem~\ref{th:an-bracket}), multiplying the (universal) toroidal pseudo bracket  with the writhe correction factor $\left( -A^{-3}\right) ^{w(L)}$ ensures invariance under Reidemeister moves R1, while retaining invariance under regular isotopy, since the writhe remains invariant under regular isotopy. For the invariance under the move PR1, the specialization $H = 1-Vd$ must be forced, as argued in \cite{HD} and discussed in Subsection~\ref{sec:planar-Jones}. So, the polynomials $P$ and $P^r$ are isotopy invariants of toroidal pseudo links, analogous to the normalized pseudo bracket for planar and annular pseudo links. 

The invariance of the normalized (reduced) ${\rm H}$-mixed pseudo bracket polynomial follows analogously from the definition of  (reduced) ${\rm H}$-mixed pseudo bracket (Theorem~\ref{th:tor-bracket}), the definition of the ${\rm H}$-writhe (Definition~\ref{def:toroidal_writhe} and Theorem~\ref{th:pl-bracket} for planar pseudo links. 
\end{proof}

\begin{definition}
\sloppy Applying the variable substitution $A= t^{-1/4}$ in the polynomials $P_L(A, V, s_{p,q,k})$, $P_L(A, V, s_{p,q})$ and $P^r_L(A, V, x,y)$ we obtain the equivalent {\it (universal) toroidal pseudo Jones polynomials}, respectively denoted by $P_L(t^{-1/4}, V, s_{p,q,k})$, $P_L(t^{-1/4}, V, s_{p,q})$ and $P^r_L(t^{-1/4}, V, x,y)$, which are isotopy invariants of toroidal pseudo links and their corresponding pseudo links in the thickened torus. Similarly, the same variable substitution in the equivalent normalized ${\rm H}$-mixed pseudo bracket polynomials $P_{{\rm H}\cup L}(A,V,s_{p,q,k})$, $P_{{\rm H}\cup L}(A,V,s_{p,q})$ and the reduced version $P^r_{{\rm H}\cup L}(A,V,x,y)$ gives rise to the \textit{${\rm H}$-mixed pseudo Jones polynomials}, denoted $P_{{\rm H}\cup L}(t^{-1/4},V,s_{p,q,k})$, $P_{{\rm H}\cup L}(t^{-1/4},V,s_{p,q})$ and $P^r_{{\rm H}\cup L}(t^{-1/4},V,x,y)$, isotopy invariants of ${\rm H}$-mixed pseudo links and their spatial analogues. 
\end{definition}

\begin{remarks}
\,\\
(a) As in the case of planar and annular pseudo links, the  variable $V$ witnesses the presence of pseudo crossings. In the absence of precrossings, the polynomials $P_L(A, V, s_{p,q,k})$, $P_L(A, V, s_{p,q})$, $P^r_L(A, V, x,y)$, $P_{{\rm H}\cup L}(A,V,s_{p,q,k})$, $P_{{\rm H}\cup L}(A,V,s_{p,q})$ and $P^r_{{\rm H}\cup L}(A,V,x,y)$ specialize to the analogues of the normalized  bracket (or the Jones polynomial for $A= t^{-1/4}$) for links in the thickened torus (for other versions of the bracket and Jones polynomials for classical links in thickened surfaces, we refer to \cites{Boden2019, Zenkina}). 
\smallbreak
\noindent (b) The variables $s_{p,q}$ (resp. $s_{p,q,k}$) in the (universal) toroidal pseudo bracket resp. the (universal) ${\rm H}$-mixed pseudo bracket (and variables $x,y$ in the reduced versions)   remain unchanged in the normalized  polynomials $P_L(A, V, s_{p,q,k})$, $P_L(A, V, s_{p,q})$, $P^r_L(A, V, x,y)$, $P_{{\rm H}\cup L}(A,V,s_{p,q,k})$, $P_{{\rm H}\cup L}(A,V,s_{p,q})$ and $P^r_{{\rm H}\cup L}(A,V,x,y)$, and they distinguish essential toroidal pseudo links from non-essential ones (i.e. ones that can be isotoped in an inclusion disc). Clearly, for non-essential toroidal pseudo links, the above normalized polynomials specialize to the corresponding planar one, $P_L(A, V)$.
\smallbreak
\noindent (c) The variables $s_{p,q}$ (resp. $s_{p,q,k}$)  in the normalized  (universal) polynomials $P_L(A, V, s_{p,q,k})$, $P_L(A, V, s_{p,q})$ and $P_{{\rm H}\cup L}(A,V,s_{p,q,k})$, $P_{{\rm H}\cup L}(A,V,s_{p,q})$  (and variables $x,y$ in the reduced versions $P^r_L(A, V, x,y)$ and $P^r_{{\rm H}\cup L}(A,V,x,y)$)  can also distinguish essential toroidal pseudo links from annular ones.  
\end{remarks}

\end{document}